\documentclass[11pt,reqno]{amsart}

\usepackage[T1]{fontenc}
\usepackage{lmodern}
\usepackage{microtype}
\usepackage{textcomp}
\usepackage{amsmath,amssymb,amsthm,mathtools}
\usepackage{bm}
\usepackage{enumitem}
\usepackage{float}
\usepackage{needspace}
\usepackage{tikz}
\usetikzlibrary{arrows.meta,positioning,calc,fit}
\usepackage{xcolor}
\usepackage[colorlinks=true,linkcolor=blue!50!black,citecolor=blue!50!black,urlcolor=blue!60!black]{hyperref}
\usepackage[nameinlink,capitalise]{cleveref}

\allowdisplaybreaks
\setlist{itemsep=2pt,topsep=4pt}

\definecolor{deepblue}{RGB}{31,65,105}
\definecolor{softgray}{RGB}{245,247,249}
\definecolor{softblue}{RGB}{238,244,251}
\definecolor{coregreen}{RGB}{219,239,222}
\definecolor{coregreenline}{RGB}{42,122,69}
\definecolor{frontorange}{RGB}{255,229,194}
\definecolor{frontline}{RGB}{191,98,24}
\definecolor{collarviolet}{RGB}{235,228,247}

\tikzset{
  nipl/.style={line width=.6pt,draw=black!70},
  nipa/.style={-{Latex[length=2mm,width=1.3mm]},line width=.6pt,draw=black!70},
  nipbox/.style={draw=black!65,fill=gray!7,rounded corners=1.5pt,
    inner sep=4pt,align=center}
}

\newtheorem{theorem}{Theorem}
\newtheorem{proposition}{Proposition}[section]
\newtheorem{lemma}{Lemma}[section]
\newtheorem{corollary}{Corollary}[section]
\theoremstyle{definition}
\newtheorem{definition}{Definition}[section]
\newtheorem*{openquestion}{Question}
\theoremstyle{remark}
\newtheorem{remark}{Remark}[section]

\newcommand{\D}{\mathbb D}
\newcommand{\T}{\mathbb T}
\newcommand{\R}{\mathbb R}
\newcommand{\Z}{\mathbb Z}
\newcommand{\N}{\mathbb N}
\newcommand{\Per}{\operatorname{Per}}
\newcommand{\rot}{\operatorname{rot}}
\newcommand{\id}{\operatorname{id}}
\newcommand{\normT}[1]{\lVert #1\rVert_{\T}}
\newcommand{\Bclass}{\mathcal B}
\newcommand{\Bclassbeta}[1]{\mathcal B_{#1}}
\newcommand{\omegaZero}{\omega_0}

\title[Non-rigid pseudo-rotations]{Non-rigid disk pseudo-rotations in $C^1$ and $C^{1,\beta}$}

\hypersetup{pdftitle={Non-rigid Irrational Pseudo-Rotations of the Disk: C1 and Holder Counterexamples to C0-Rigidity},pdfauthor={Jian Wang and Peizheng Yu},pdfsubject={Area-preserving disk pseudo-rotations, counterexamples to C0-rigidity, and Holder regularity}}

\author{Ziran Liu}
\address{Shanghai Institute for Mathematics and Interdisciplinary Sciences (SIMIS), Shanghai 200433, China}
\address{Research Institute of Intelligent Complex Systems, Fudan University, Shanghai 200433, China}
\email{zliu@simis.cn}

\author{Jian Wang}
\thanks{Jian Wang was partially supported by NSFC (No. 12361141812) and the Fundamental Research Funds for the Central University (No. 63213122).}
\address{School of Mathematical Sciences and LPMC, Nankai University, Tianjin 300071, PR China}
\email{wangjian@nankai.edu.cn}

\author{Peizheng Yu}
\address{School of Mathematics, Shandong University, Jinan 250100, PR China}
\email{yupzh@sdu.edu.cn}

\date{}


\begin{document}

\begin{abstract}
We construct area-preserving irrational pseudo-rotations of the closed disk
that fail to be $C^0$-rigid.  For every Brjuno rotation number we obtain
$C^1$ examples, and they may be chosen arbitrarily $C^1$-close to the rigid
rotation with the same rotation number.  For $0<\beta<1$, a weighted Brjuno
condition gives $C^{1,\beta}$ examples that are arbitrarily close in the
$C^{1,\beta}$ topology; this weighted range contains every rotation number
with denominator-growth exponent strictly smaller than $1/\beta$ and also
some numbers at the critical exponent. For a suitable Brjuno--Liouville rotation number, we also obtain an example whose derivative growth is not bounded by any polynomial.
\end{abstract}

\maketitle
\raggedbottom

\section{Introduction}

Let $\D=\{z\in\R^2:|z|\le1\}$ with the standard area form $\omegaZero=dX\wedge dY$, and write $\mathcal R_\alpha(z)=e^{2\pi i\alpha}z$ for the rigid rotation.  An irrational disk pseudo-rotation is called $C^0$-rigid if some sequence of iterates converges uniformly to the identity.  Several positive rigidity results are known at higher regularity.  Bramham established $C^0$-rigidity for smooth pseudo-rotations whose rotation number is sufficiently Liouvillean \cite{BramhamRigid}.  Avila--Fayad--Le Calvez--Xu--Zhang proved $C^{k-1}$-rigidity for $C^k$ pseudo-rotations, $k\ge2$, with non-Brjuno rotation number \cite{AFLXZ}.  See \cite{AFLXZ,BramhamZhangSurvey,FayadKrikorian} for broader discussions of the rigidity problem.

Bramham asked whether every irrational pseudo-rotation is $C^0$-rigid \cite[Question~2]{BramhamRigid}; in that paper, an irrational pseudo-rotation is assumed to be $C^\infty$.  Bramham's question was subsequently restated in the following regularity-sensitive form by Avila--Fayad--Le Calvez--Xu--Zhang \cite[Question~3]{AFLXZ} and was also recorded by Fayad--Krikorian \cite[Question~13]{FayadKrikorian}:
\begin{center}
	\emph{Is every $C^k$ pseudo-rotation $f$ $C^0$-rigid?}
\end{center}
Here $C^k$ denotes the regularity class: finite integers $k\ge1$ give the finite-differentiability cases, whereas $C^\infty$ and $C^\omega$ denote the smooth and real-analytic categories.  

At the purely topological level, Avila--Fayad--Le Calvez--Xu--Zhang gave, for every irrational $\alpha$, an area-preserving $C^0$ pseudo-rotation of the disk with rotation number $\alpha$ that is not $C^0$-rigid \cite[p.~681]{AFLXZ}.  Their construction places a Denjoy circle homeomorphism on one invariant circle and then uses the Oxtoby--Ulam theorem to conjugate the resulting full-support invariant measure to Lebesgue area.  Thus non-rigidity is already possible for area-preserving homeomorphisms; the issue addressed here is whether it persists under differentiable regularity.

\Needspace{3\baselineskip}
Theorem~\ref{thm:main}\textnormal{(i)} below gives a negative answer, for every
Brjuno rotation number, to the endpoint case $k=1$ of
\cite[Question~3]{AFLXZ}, equivalently
\cite[Question~13]{FayadKrikorian}.
At the level of arithmetic classes, these counterexamples occupy the Brjuno
side complementary to the non-Brjuno rigidity theorem of
Avila--Fayad--Le Calvez--Xu--Zhang for $C^k$ pseudo-rotations with $k\ge2$.
In their proof, the $C^2$ regularity requirement enters through an effective
finite-information version of Katok's closing lemma for area-preserving
surface diffeomorphisms, which assumes a bound on second derivatives
\cite[Theorem~4 and the proof of Theorem~1]{AFLXZ}.
The regularities are different, so this comparison does not assert a
$C^1$ arithmetic dichotomy.  The $C^{1,\beta}$ statements below strengthen
the endpoint construction while remaining strictly below $C^2$; they do
not settle the finite-differentiability cases $k\ge2$, the smooth case,
or the real-analytic case of the cited rigidity question.

\subsection{Definitions and conventions}

We write $\T=\R/\Z$ and denote by $[x]=x+\Z\in\T$ the class of
$x\in\R$.  We also write $R_\alpha(x)=x+\alpha$ on $\T$ and
$\normT{x}=\min_{p\in\Z}|x-p|$ for $x\in\R$. All logarithms are natural unless the base is explicitly indicated. Rotation numbers are taken modulo $\Z$, with their
representatives in $(0,1)$ used when convenient.   For positive sequences, $u_n\asymp v_n$
means that $c\,v_n\le u_n\le C\,v_n$ for all sufficiently large $n$, for
some $0<c\le C<\infty$ independent of $n$; these constants may depend on
fixed parameters such as $\beta$.  This denotes two-sided comparability,
not convergence of $u_n/v_n$ to one.  We use $u_n=O(v_n)$ when
$|u_n|\le C v_n$ eventually, with the same convention on dependencies;
$o(1)$ denotes a quantity tending to zero.

Differences of circle maps are taken in compatible degree-one lifts.
Flow lifts start at the identity, cumulative lifts are defined by
composition, and conjugated lifts use $\widetilde R_\alpha(x)=x+\alpha$.
For homeomorphisms, area preservation means preservation of Lebesgue
measure on $\D$.

\begin{definition}\label{def:pseudorotation}
Let $F:\D\to\D$ be an orientation- and area-preserving homeomorphism.  We call $F$ an \emph{irrational pseudo-rotation} if
\[
 F(0)=0,
 \qquad
 \Per(F):=\{z\in\D:F^n(z)=z\text{ for some }n\ge1\}=\{0\}.
\]
To define its rotation number, set $f=F|_{\partial\D}$, let
$\pi:\R\to\partial\D$ be given by $\pi(t)=e^{2\pi i t}$, and choose a lift
$\widetilde f:\R\to\R$ satisfying
$\pi\circ\widetilde f=f\circ\pi$ and
$\widetilde f(t+1)=\widetilde f(t)+1$.  This lift is an increasing
homeomorphism.  Then
\[
 \rot(F):=
 \left[
 \lim_{n\to\infty}
 \frac{\widetilde f^{\,n}(t)-t}{n}
 \right]\in\T
\]
is independent of $t$ and of the chosen lift modulo $\Z$.  Since
$\Per(F)=\{0\}$, the boundary map $f$ has no periodic point; hence
$\rot(F)\in\T\setminus\mathbb Q$.
\end{definition}

For a disk pseudo-rotation, this boundary rotation number agrees with the unique rotation number on $\D\setminus\{0\}$; see, for example, the discussion and references in \cite{AFLXZ,BramhamZhangSurvey}.  For the maps constructed below, the common value can be read directly from every invariant circle and from the boundary.

\begin{definition}\label{def:c0rigid}
A homeomorphism $F$ of a compact metric space is \emph{$C^0$-rigid} if there exists a sequence $n_j\to\infty$ such that $F^{n_j}\to\id$ uniformly.
\end{definition}

A $C^1$ pseudo-rotation is understood to be a diffeomorphism whose map
and inverse are $C^1$ up to the boundary.  For $0<\beta\le1$, a
diffeomorphism of the closed disk is said to be $C^{1,\beta}$ if both the map and its inverse are $C^1$ up to the boundary and their first derivatives are $\beta$-H\"older on the closed disk.  For $0<\beta<1$ this is the regularity class often denoted $C^{1+\beta}$.

\Needspace{8\baselineskip}
For a $C^1$ map $u:\D\to\R^2$ whose derivative is
$\beta$-H\"older, we use the norm
\[
\|u\|_{C^{1,\beta}(\D)}
:=
\sup_{z\in\D}|u(z)|
+
\sup_{z\in\D}\|Du(z)\|
+
[Du]_{C^{0,\beta}(\D)},
\]
where
\[
[Du]_{C^{0,\beta}(\D)}
:=
\sup_{\substack{z,w\in\D\\ z\neq w}}
\frac{\|Du(z)-Du(w)\|}{|z-w|^\beta}.
\]
Here $|\cdot|$ denotes the Euclidean norm on $\R^2$ and
$\|\cdot\|$ the induced operator norm on linear maps.  We also write
\[
\|u\|_{C^1(\D)}
:=\sup_{z\in\D}|u(z)|+\sup_{z\in\D}\|Du(z)\|.
\]

The batching of successive squeezing scales is measured by the Brjuno sum.
For positive H\"older exponents, the same construction leads to the following
weighted refinement.

\begin{definition}\label{def:arith-classes}
Let $p_n/q_n$ be the continued-fraction convergents of an irrational $\alpha$.  We set
\[
 B(\alpha):=\sum_{n\ge0}\frac{\log q_{n+1}}{q_n},
 \qquad
 \Bclass:=\{\alpha\in\T\setminus\mathbb Q:B(\alpha)<\infty\}.
\]
Thus $\Bclass$ is the Brjuno class.  For $0<\beta<1$, we also set
\begin{equation}\label{eq:weighted-Brjuno-class}
 B_\beta(\alpha)
 :=\sum_{n\ge0}
 \frac{q_{n+1}^{\beta}(1+\log q_{n+1})^{2\beta}}{q_n},
 \qquad
  \Bclassbeta{\beta}:=\{\alpha\in\T\setminus\mathbb Q:
  B_\beta(\alpha)<\infty\}.
\end{equation}
\end{definition}

\Needspace{20\baselineskip}
\subsection{Main results}

The first result gives the global regularity ranges of the construction.

\begin{theorem}\label{thm:main}
Let $\alpha\in\Bclass$.
\begin{enumerate}[label=\textnormal{(\roman*)}]
\item There exists an orientation- and
area-preserving $C^1$ irrational pseudo-rotation
$F_\alpha:\D\to\D$ which is not $C^0$-rigid.
\item Fix $0<\beta<1$.  If $\alpha\in\Bclassbeta{\beta}$, then the
same conclusion holds with $F_\alpha$ of class $C^{1,\beta}$.
\end{enumerate}
In both cases $\Per(F_\alpha)=\{0\}$, the boundary restriction is a genuine Denjoy diffeomorphism of rotation number $\alpha$, and $\operatorname{int}(\D)\setminus\{0\}$ is foliated by smooth invariant circles on which the dynamics is smoothly conjugate to $R_\alpha$.
\end{theorem}

\begin{remark}\label{rem:rotation-range}
The hypotheses are nested: $\Bclassbeta{\beta}\subset\Bclass$, since
$\log x\le C_\beta x^\beta$ for $x\ge2$.  The inclusion is strict, and
one may see this by choosing the partial quotients so that
$q_{n+1}\asymp q_n^{1/\beta}$.  The positive H\"older condition also has a
genuine critical boundary.  If
\[
 \nu(\alpha):=\limsup_{n\to\infty}
 \frac{\log q_{n+1}}{\log q_n}<\frac1\beta,
\]
then $\alpha\in\Bclassbeta{\beta}$, as proved in
Lemma~\ref{lem:nu-growth}\textnormal{(i)}.  The limsup is taken over sufficiently
large indices, for which $q_n>1$.  The same membership can still hold
when $\nu(\alpha)=1/\beta$; for example, one may choose the continued-fraction
coefficients so that
\[
 q_{n+1}\asymp q_n^{1/\beta}
 \exp\!\bigl(- (\log q_n)^{3/4}\bigr).
\]
On the other hand, $B_\beta(\alpha)<\infty$ forces
$q_{n+1}^{\beta}/q_n\to0$ and hence $\nu(\alpha)\le1/\beta$.
Moreover, Lemma~\ref{lem:nu-growth}\textnormal{(ii)} gives
$\nu(\alpha)=1$ for Lebesgue almost every irrational $\alpha$.
Since $1<1/\beta$ for every $0<\beta<1$, the same full-measure set lies in
$\Bclassbeta{\beta}$ for every such $\beta$, giving the arithmetic input for
Corollary~\ref{cor:local-nonrigidity-nu} below.
The Brjuno class itself is much larger than the finite-type classes and
contains Liouville numbers; an explicit example is given in
Remark~\ref{rem:brjuno-liouville} below.
\end{remark}

Two high-regularity conjugacy results provide a comparison for the local
non-rigidity theorem below.

In the smooth category, a consequence of Herman's last geometric
theorem gives, for each fixed Diophantine class, a suitable finite
integer $s$ with the following property.  Let $\alpha$ belong to that
class, and let $G$ be a $C^\infty$ diffeomorphism of the closed annulus
that is isotopic to the identity, preserves both boundary components,
has rotation number $\alpha$ on one boundary component, satisfies the
intersection property, and has only finitely many periodic points in
the open annulus.  If $G$ is sufficiently $C^s$-close to
$S_\alpha(\theta,r)=(\theta+\alpha,r)$, then it is
$C^\infty$-conjugate to $S_\alpha$
\cite[Corollary~1]{FayadKrikorianHerman}.

In the real-analytic category, the relevant Brjuno theorem is due to
R\"ussmann.  As recorded in \cite[p.~678, \S1.3]{AFLXZ}, R\"ussmann's
analytic Brjuno result \cite{Russmann}, combined with the globalization
argument used in \cite[Corollary~1 and its proof]{FayadKrikorianHerman},
implies that, for a Brjuno rotation number $\alpha$ and a fixed analyticity
width, every sufficiently close real-analytic pseudo-rotation of rotation
number $\alpha$ is analytically conjugate to $\mathcal R_\alpha$, and hence
is $C^\omega$-rigid.

In contrast, Theorem~\ref{thm:local-nonrigidity} gives non-$C^0$-rigid
examples arbitrarily close to the rotation in $C^1$ or $C^{1,\beta}$,
with the admissible rotation number fixed.  Since
$\Bclassbeta{\beta}\subset\Bclass$, this contrast occurs within the
Brjuno range.  These examples are not topologically conjugate to the
rotation on the closed disk.

\Needspace{18\baselineskip}
\begin{theorem}\label{thm:local-nonrigidity}
Let $\alpha\in\Bclass$.
\begin{enumerate}[label=\textnormal{(\roman*)}]
\item For every $\varepsilon>0$ there exists an
area-preserving $C^1$ irrational pseudo-rotation $F$ of $\D$ such that
\[
 \rot(F)=\alpha,
 \qquad
 \|F-\mathcal R_\alpha\|_{C^1(\D)}<\varepsilon,
\]
while $F$ is not $C^0$-rigid.
\item Fix $0<\beta<1$.  If $\alpha\in\Bclassbeta{\beta}$, then for every
$\varepsilon>0$ there exists an area-preserving $C^{1,\beta}$ irrational
pseudo-rotation $F$ of $\D$ such that
\[
 \rot(F)=\alpha,
 \qquad
 \|F-\mathcal R_\alpha\|_{C^{1,\beta}(\D)}<\varepsilon,
\]
while $F$ is not $C^0$-rigid.
\end{enumerate}
\end{theorem}

In particular, we have the following local consequence in terms of $\nu(\alpha)$.

\begin{corollary}\label{cor:local-nonrigidity-nu}
Let $\alpha\in\T\setminus\mathbb Q$ and fix $0<\beta<1$.  If $\nu(\alpha)<1/\beta$, then the conclusion of
Theorem~\ref{thm:local-nonrigidity}\textnormal{(ii)} holds.  In particular,
the common full-measure set
$E:=\{\alpha\in\T\setminus\mathbb Q:\nu(\alpha)=1\}$ has the following
property: for every $\alpha\in E$, the conclusion holds for every
$0<\beta<1$.
\end{corollary}

\begin{remark}\label{rem:kam-comparison}
By the corollary above, the $C^{1,\beta}$ analogue of the semi-local rigidity questions
\cite[Question~5]{AFLXZ} and \cite[Question~14]{FayadKrikorian} fails
even when the neighborhood is allowed to depend on $\alpha$.
Their original analytic versions concern a fixed analyticity strip
and are not settled here.  In the analytic category,
Avila--Fayad--Le Calvez--Xu--Zhang prove that, for every irrational
$\alpha$ and every fixed analyticity width $\delta>0$, there is
$\varepsilon(\delta,\alpha)>0$ such that a real-analytic pseudo-rotation
$f$ with $\rot(f)=\alpha$ and
$|f-\mathcal R_\alpha|_\delta\le\varepsilon(\delta,\alpha)$ is
$C^\infty$-rigid \cite[Corollary~A]{AFLXZ}.  This covers all irrational
rotation numbers, whereas the Brjuno statement recalled above has the
stronger conclusion of analytic conjugacy.
\end{remark}

\begin{remark}\label{rem:denjoy-boundary-obstruction}
For every $0<\beta<1$, the one-dimensional Denjoy theory provides a useful comparison: there are $C^{1,\beta}$ Denjoy
diffeomorphisms with any prescribed irrational rotation number, and their
wandering intervals preclude $C^0$-rigidity \cite{Herman}.  The disk setting,
however, has a genuinely two-dimensional obstruction.  If $\alpha$ is
super-Liouville in the sense that
\[
 \limsup_{n\to\infty}\frac{\log q_{n+1}}{q_n}=+\infty,
\]
then Theorem~2 of Avila--Fayad--Le Calvez--Xu--Zhang implies that every
Lipschitz disk pseudo-rotation of rotation number $\alpha$ is $C^0$-rigid
\cite[Theorem~2]{AFLXZ}.  Hence such a Denjoy diffeomorphism cannot occur
as the boundary restriction of a $C^1$ area-preserving disk pseudo-rotation:
rigidity would restrict to the boundary, contradicting its wandering
interval.  Thus, in the disk problem, area preservation, arithmetic, and
regularity act jointly; our conservative realization is not a formal lift of
the one-dimensional Denjoy construction.  This is consistent with
Theorem~\ref{thm:main}, since the Brjuno condition implies
$\log q_{n+1}/q_n\to0$.

\end{remark}

\subsection{A derivative-growth consequence and an open problem}\label{sec:consequences}

The following $C^1$ consequence concerns the derivative-growth problem
raised in \cite[Question~4]{AFLXZ} and
\cite[Question~15]{FayadKrikorian}, whose original smooth versions remain
unresolved by this construction.
For an iterate, we write
$\|DF^m\|_{C^0(\D)}=\sup_{z\in\D}\|D(F^m)(z)\|$.

\Needspace{10\baselineskip}
\begin{corollary}\label{cor:superpolynomial-growth}
There exist a Brjuno--Liouville irrational number $\alpha$ and an
area-preserving $C^1$ irrational pseudo-rotation $F:\D\to\D$ with
$\rot(F)=\alpha$ such that $F$ is not $C^0$-rigid and, for every $A>0$,
\[
 \limsup_{m\to\infty}
 \frac{\|DF^m\|_{C^0(\D)}}{m^A}=+\infty.
\]
\end{corollary}

The proof is given at the end of Section~\ref{sec:arithmetic-schedules}.
It combines the Brjuno--Liouville example in
Remark~\ref{rem:brjuno-liouville} with the displacement estimate
\cite[Corollary~B]{AFLXZ}. 

Our construction requires a finite Brjuno sum, whereas the
super-Liouville range is rigid even for Lipschitz pseudo-rotations, as
explained in Remark~\ref{rem:denjoy-boundary-obstruction}.  This motivates the following question in the intermediate arithmetic
range at $C^1$ regularity.


\Needspace{9\baselineskip}
\begin{openquestion}
Let $\alpha$ be a non-Brjuno irrational number which is not
super-Liouville, so that
\[
 B(\alpha)=+\infty,
 \qquad
 \limsup_{n\to\infty}\frac{\log q_{n+1}}{q_n}<+\infty.
\]
Must every area-preserving $C^1$ irrational pseudo-rotation of the disk
with rotation number $\alpha$ be $C^0$-rigid, or can one construct such
a pseudo-rotation which is not $C^0$-rigid?
\end{openquestion}

We intend to investigate this question in subsequent work.


\subsection{Construction and proof strategy}\label{sec:proof-strategy}

The global and local results use the same Hamiltonian construction; the
arithmetic of the rotation number determines the allocation of scales and
the available summability.

The main difficulty is to realize the Denjoy boundary dynamics of
Theorem~\ref{thm:main} by an area-preserving disk diffeomorphism while
retaining differentiability in the normal direction.  We outline the
protected-core and moving-front mechanism below.

Section~\ref{sec:squeeze} constructs protected micro-squeezes
$L_{r,t}$ on the dyadic scales $\delta_r$.  Lemma~\ref{lem:batched-squeeze}
combines them into block composites with protection and exact
cancellation within each scale phase; see
Figure~\ref{fig:protected-phase-order}.

Section~\ref{sec:towers} places these composites on finite rotation towers
and allocates $d_{m,r}$ edges of each active front to scale $r$.
Lemma~\ref{lem:moving-front} reduces the defect on each active
source--target pair to flow time $\pm d_{m,r}^{-1}$; see
Figure~\ref{fig:batched-active-edges}.

\Needspace{6\baselineskip}
In Section~\ref{sec:density}, Proposition~\ref{prop:protected-density}
bounds the rotation-coordinate density from below on the active fronts.
Lemma~\ref{lem:stabilization} gives pairwise disjoint nondegenerate
intervals $I_j$ with $g_m(I_j)=I_{j+1}$ at late stages, producing
wandering intervals of the boundary limit; the stabilization mechanism
is illustrated in Figure~\ref{fig:fixed-pullback-density}.

Section~\ref{sec:hamiltonian} realizes the circle construction by
Hamiltonian maps in shrinking boundary collars.
The cotangent formula of Lemma~\ref{lem:ham-squeeze}, valid along orbit
segments in the exact zone, and the nested-zone control of
Lemma~\ref{lem:nested-cotangent} provide the coordinates used in the
derivative estimates; see Figure~\ref{fig:nested-exact-zones}.
The conjugates \eqref{eq:2d-HF} have the prescribed circle boundary
restrictions.

The quantitative estimates are established in
Section~\ref{sec:master-ledger}.  The natural-scale bounds of
Lemma~\ref{lem:hk-higher-variable} and the active-region cotangent
estimates of Lemma~\ref{lem:cotangent-derivatives-variable} yield
Proposition~\ref{prop:master-ledger}: on the physical preimage of each
active scale-$r$ source rectangle, the increment $\Delta_m=F_{m+1}-F_m$
satisfies
\[
	\begin{aligned}
		\|\Delta_m\|_{C^0}\le C_{\Delta,0}\frac{\delta_r}{d_{m,r}},
		\quad
		\|D\Delta_m\|_{C^0}\le C_{\Delta,1}\frac1{d_{m,r}},\quad\text{and}\quad
		\|D^2\Delta_m\|_{C^0}\le
		C_{\Delta,2}\frac{\delta_r^{-1}}{d_{m,r}}.
	\end{aligned}
\]
Corollary~\ref{cor:holder-stage}
interpolates the global component zero extensions and combines them by
\eqref{eq:holder-zero-gluing}.  With the macro-stage error budgets
$\Xi_m(\eta)$ defined in \eqref{eq:Xi-definition}, it gives
\[
 \|\Delta_m\|_{C^1}\le C_{\Delta,01}\Xi_m(0),
 \qquad [D\Delta_m]_{C^{0,\beta}}
 \le C_{\mathrm{Hol},\beta}\Xi_m(\beta).
\]

\Needspace{6\baselineskip}
Section~\ref{sec:master-convergence} turns these estimates into convergence
and disk completion.  Proposition~\ref{prop:master-convergence} gives
$C^1$ convergence under $\sum_m\Xi_m(0)<\infty$, and $C^{1,\beta}$
regularity of both the limit and its inverse under
$\sum_m\Xi_m(\beta)<\infty$.
The conjugacies stabilize away from the outer boundary, and a fixed rigid
inner collar permits attachment of a central rotating disk.

\Needspace{7\baselineskip}
Finally, Section~\ref{sec:arithmetic-schedules} defines the checkpoint
scales \eqref{eq:checkpoints}, and
Proposition~\ref{prop:brjuno-batched-schedule} supplies the tower data.
The unweighted allocation $(\eta=0)$ makes the block cost
bounded by a constant multiple of $(1+\log q_{j+1})/q_j$ and proves
the Brjuno $C^1$ result.
Integer allocation weighted by $\delta_r^{-\beta}$ gives the weighted cost in
\eqref{eq:weighted-Brjuno-class}.  Starting at a late checkpoint makes its
tail arbitrarily small and proves Theorem~\ref{thm:local-nonrigidity};
Lemma~\ref{lem:nu-growth} gives Corollary~\ref{cor:local-nonrigidity-nu}.

\Needspace{9\baselineskip}
Throughout the construction, all constants are positive.  We use $m$ for a
macro-stage and $r$ for a micro-scale.  The letter $j$ denotes a tower
level locally and a continued-fraction index in
Section~\ref{sec:arithmetic-schedules}.
Descriptive subscripts distinguish constants from maps and parameters.
Constants remain fixed once chosen, except for explicitly indicated
enlargements, which may retain the same notation without changing their
permitted dependencies.  Unless explicitly stated otherwise, estimate
constants are independent of the macro-stage $m$, the micro-scale $r$,
the block widths, the tower heights, and every later admissible increase
of $K$.  A subscript $\beta$ permits dependence on the H\"older exponent;
arithmetic dependencies are specified where they arise.

\flushbottom
\section{Variable protected squeezing models}\label{sec:squeeze}

The local model has a fixed contracting core and a stage-dependent outer transition.  The latter is lengthened so that the maximal expansion tends to one while the contracting part remains unchanged.

We work on
\[
 P=[-2,2],\qquad P^+=[-1,1].
\]

The following lemma provides the required squeeze with estimates uniform in the permitted outer expansion.

\begin{lemma}\label{lem:protected-squeeze}
Set $\kappa_{\mathrm{sq}}=e^{-3}$.  There is a fixed constant $C>0$ and there are constants
$C_{\mathrm{sq},\ell}>0$, $0\le\ell\le5$, all independent of $\varepsilon$, with the following property.  For every sufficiently small $\varepsilon>0$, set
\[
 A_\varepsilon:=C\varepsilon^{-1}+3>4.
\]
Then there exist a smooth odd vector field $V_\varepsilon$ on $(-A_\varepsilon,A_\varepsilon)$, vanishing near the boundary, and its flow $\Phi_\varepsilon^t$ such that, for $0\le t\le1$,
\begin{enumerate}[label=\textnormal{(S\arabic*)}]
\item $\Phi_\varepsilon^1(P)=P^+$;
\item $\kappa_{\mathrm{sq}}\le D\Phi_\varepsilon^t\le e^{\varepsilon/2}$;
\item $D\Phi_\varepsilon^t(z)<1$ implies $\Phi_\varepsilon^t(z)\in P^+$;
\item $\|V_\varepsilon\|_{C^\ell}\le C_{\mathrm{sq},\ell}$ for $1\le\ell\le5$, and for $|u|\le1$,
\[
 \|\Phi_\varepsilon^u-\id\|_{C^0}\le C_{\mathrm{sq},0}|u|,
 \qquad
 \|D\Phi_\varepsilon^u-1\|_{C^0}\le C_{\mathrm{sq},1}|u|,
\]
\[
 \|D^\ell\Phi_\varepsilon^u\|_{C^0}\le C_{\mathrm{sq},\ell}|u|,
 \qquad 2\le\ell\le5.
\]
\end{enumerate}
We refer to the fixed choices made in this construction, and the
resulting uniform squeeze bounds, as the fixed one-dimensional
squeeze data.
\end{lemma}

\begin{proof}
We choose once and for all a smooth function $\mu:[0,1]\to[0,3]$ which equals one near $0$, vanishes near $1$, and satisfies $\int_0^1\mu=1$.  We put
\[
 \psi(x)=\int_0^x\mu(s)\,ds,
\]
and define, on the positive half-line,
\[
 V_\varepsilon(x)=-\psi(x)\quad(0\le x\le1),
 \qquad
 V_\varepsilon(x)=-1\quad(1\le x\le2).
\]
We choose a fixed smooth nondecreasing cutoff $\zeta:\R\to[0,1]$ with $\zeta=0$ on $(-\infty,0]$ and $\zeta=1$ on $[1,\infty)$.  We then choose
$C>2\|\zeta'\|_\infty$ once and for all and set
$L_\varepsilon=C\varepsilon^{-1}$.  Thus
$\|\zeta'\|_\infty/L_\varepsilon\le\varepsilon/2$.  On the outer transition we set
\[
 V_\varepsilon(x)=-1+\zeta\!\left(\frac{x-2}{L_\varepsilon}\right)
 \qquad(2\le x\le 2+L_\varepsilon).
\]
We extend this field by zero a little beyond $2+L_\varepsilon$ and then extend it oddly.  With $A_\varepsilon=L_\varepsilon+3$, all joins are smooth.  The inner part is independent of $\varepsilon$, while the $\ell$-th derivative of the outer transition is $O(L_\varepsilon^{-\ell})$.  Hence, after enlarging the fixed constants $C_{\mathrm{sq},\ell}$ if necessary, all stated fixed-order bounds are uniform and
\[
 -3\le V_\varepsilon'\le\varepsilon/2.
\]

The point $2$ moves with speed $-1$ until it reaches $1$, so $\Phi_\varepsilon^1(2)=1$; oddness gives the other endpoint and proves (S1).  The variational identity
\[
 D\Phi_\varepsilon^t(z)
 =\exp\!\left(\int_0^tV_\varepsilon'(\Phi_\varepsilon^r(z))\,dr\right)
\]
gives (S2).  The interval $P^+$ is forward invariant.  Outside its interior one has $V_\varepsilon'\ge0$, so if the final point is outside $P^+$ the whole orbit segment has derivative at least one.  This proves (S3).

Finally, the vector fields and their derivatives through order five are uniformly bounded.  The standard variational equations for the flow on the fixed time interval $[-1,1]$ therefore give uniform bounds for its spatial derivatives through order five.  We enlarge $C_{\mathrm{sq},0},\ldots,C_{\mathrm{sq},5}$ once more to cover these flow estimates.  Since $\Phi_\varepsilon^0=\id$ and all spatial derivatives of order at least two vanish at time zero, integration of the variational equations gives the factor $|u|$ in (S4).
\end{proof}

\subsection{The scale-dependent family}\label{sec:variable-model}

Let $\varepsilon_{\mathrm{adm}}>0$ be a threshold for
Lemma~\ref{lem:protected-squeeze}.  We choose $K_{\min}$ so large that
$\overline\lambda:=e^{1/(2K_{\min}^2)}<3/2$ and, for every integer
$r\ge K_{\min}$,
\[
 r^{-2}\le\varepsilon_{\mathrm{adm}},\qquad
 2(Cr^2+3)2^{-r}<1,
\]
and
\begin{equation}\label{eq:nested-angular-supports}
 (C(r+1)^2+3)2^{-(r+1)}<(Cr^2+3)2^{-r}.
\end{equation}
We later choose a starting scale $K\ge K_{\min}$.  For every $r\ge K$ set
\[
 \delta_r=2^{-r},\qquad
 \varepsilon_r=r^{-2},\qquad
 \lambda_r=e^{\varepsilon_r/2},
 \qquad A_r=Cr^2+3.
\]
Let $V_r=V_{\varepsilon_r}$ and $\Phi_r^t=\Phi_{\varepsilon_r}^t$.
Then
\[
 \prod_{r\ge K}\lambda_r<\infty,\qquad
 \sup_{r\ge K}\lambda_r\le\overline\lambda<\frac32,\qquad
 2A_r\delta_r<1.
\]
We fix a cotangent-zone constant $\mathfrak b$, independently of $K$, with
\begin{equation}\label{eq:b-choice}
 \frac{\overline\lambda}{2}<\mathfrak b<1;
 \qquad\text{for definiteness, take }\mathfrak b=\frac78.
\end{equation}

Fix $\theta_0\in\T$ and put
\[
 J_r=\theta_0+\delta_r(-A_r,A_r),
 \qquad P_r=\theta_0+\delta_rP.
\]
We call $J_r$ the scale-$r$ \emph{squeeze interval} and $P_r$ its
\emph{protected core}.  Equation~\eqref{eq:nested-angular-supports} gives
$\overline{J_{r+1}}\subset J_r$, while
$\theta_0+\delta_rP^+=P_{r+1}$.  We denote by $L_{r,t}$ the
rescaled copy of $\Phi_r^t$ on $J_r$, extended by the identity outside
$J_r$, and call it the scale-$r$ \emph{protected micro-squeeze}.
Flow parameters $t$ lie in $[0,1]$, and signed times $u$ lie in
$[-1,1]$.  Uniformly in $r$,
\begin{align}
 L_{r,1}(P_r)&=P_{r+1},\label{eq:full-squeeze}\\
 \kappa_{\mathrm{sq}}\le L_{r,t}'&\le\lambda_r,\label{eq:L-first}\\
 L_{r,t}'(x)<1&\Longrightarrow L_{r,t}(x)\in P_{r+1},\label{eq:L-localized}\\
 \|L_{r,u}-\id\|_{C^0}&\le C_{\mathrm{sq},0}\delta_r|u|,\label{eq:L0}\\
 \|L_{r,u}'-1\|_{C^0}&\le C_{\mathrm{sq},1}|u|,\label{eq:L1}\\
 \|L_{r,u}^{(\ell)}\|_{C^0}&\le C_{\mathrm{sq},\ell}\delta_r^{1-\ell}|u|,
 \qquad 2\le\ell\le5.\label{eq:Lhigher}
\end{align}

We now combine consecutive scales.  For integers $K\le a<b$, write
$w=b-a$ and, for $r=a,\ldots,b-1$, define
\[
 \tau_{a,b,r}(s)
 :=\min\{1,\max\{0,s-(b-1-r)\}\},\qquad 0\le s\le w.
\]
Set
\begin{equation}\label{eq:batched-circle-map}
 G_{a,b}(s)
 :=L_{b-1,\tau_{a,b,b-1}(s)}\circ\cdots\circ
 L_{a,\tau_{a,b,a}(s)}.
\end{equation}
We call $G_{a,b}$ the \emph{block composite}.  Its \emph{scale-$r$ phase}
is the parameter interval $[b-1-r,b-r]$, on which the scale-$r$ flow
runs from time zero to time one while all coarser factors are the
identity and all finer factors are at time one.

\begin{lemma}\label{lem:batched-squeeze}
The family in \eqref{eq:batched-circle-map} satisfies
\[
 G_{a,b}(0)=\id,\qquad G_{a,b}(w)(P_a)=P_b,
\]
and, for every $s\in[0,w]$,
\begin{equation}\label{eq:batched-protection}
 G_{a,b}'(s;x)<1\quad\Longrightarrow\quad G_{a,b}(s)(x)\in P_b.
\end{equation}
If $s,t$ lie in the same unit phase and that phase corresponds to the
scale $r$, then
\begin{equation}\label{eq:phase-cancellation-circle}
 G_{a,b}(t)^{-1}G_{a,b}(s)=L_{r,s-t}.
\end{equation}
Consequently, for $|s-t|\le1$ in one phase,
\[
 \|G_{a,b}(t)^{-1}G_{a,b}(s)-\id\|_{C^0}
 \le C_{\mathrm{sq},0}\delta_r|s-t|,
\]
\[
 \|D(G_{a,b}(t)^{-1}G_{a,b}(s))-1\|_{C^0}
 \le C_{\mathrm{sq},1}|s-t|.
\]
\end{lemma}

\begin{proof}
At $s=w$ every factor is at time one, so repeated use of
\eqref{eq:full-squeeze} sends $P_a$ to $P_b$.  If the derivative of the
composite is less than one, some participating factor has derivative less
than one.  By \eqref{eq:L-localized}, its image lies in the next protected
core.  All finer factors following it are already at time one, and hence
carry this core successively to $P_b$.  This proves
\eqref{eq:batched-protection}.  More explicitly, in the phase
$[b-1-r,b-r]$ all coarser factors are the identity and all finer factors
are at time one.  With
$\mathcal L_r=L_{b-1,1}\circ\cdots\circ L_{r+1,1}$ (the identity if $r=b-1$),
we have $G_{a,b}(s)=\mathcal L_r\circ L_{r,s-(b-1-r)}$ throughout that phase.
Thus the common factor $\mathcal L_r$ cancels in
$G_{a,b}(t)^{-1}G_{a,b}(s)$, and the flow group law gives
\eqref{eq:phase-cancellation-circle}.  This also covers phase endpoints.
The estimates follow from \eqref{eq:L0}--\eqref{eq:L1}.
\end{proof}

\Needspace{18\baselineskip}
Figure~\ref{fig:protected-phase-order} distinguishes parameter activation
from the order in which the full composite acts on a point.

\begin{figure}[H]
\centering
\begin{tikzpicture}[x=.97cm,y=.85cm,font=\small,>=Latex]
  \node[anchor=west,font=\bfseries] at (0,4.8) {(a) Protected cores};
  \draw[nipl,fill=gray!7] (.2,3.55) rectangle (6.2,4.15);
  \draw[nipl,fill=coregreen] (2,3.55) rectangle (4.4,4.15);
  \node at (.8,3.85) {$J_r$};
  \node at (3.2,3.85) {$P_r$};
  \draw[nipa] (3.2,3.45) -- (3.2,2.8);
  \node[anchor=west] at (3.45,3.1) {$L_{r,1}|_{P_r}$};
  \draw[nipl,fill=gray!7] (1,2.1) rectangle (5.4,2.7);
  \draw[nipl,fill=coregreen] (2.6,2.1) rectangle (3.8,2.7);
  \node at (1.65,2.4) {$J_{r+1}$};
  \node at (3.2,2.4) {$P_{r+1}$};
  \node[align=center] at (3.2,1.35)
    {$L_{r,t}'(x)<1$\\$\Longrightarrow\ L_{r,t}(x)\in P_{r+1}$};
  \node[anchor=west,font=\bfseries] at (7.5,4.8) {(b) Two orders ($b=a+3$)};
  \foreach \x/\lab in {7.5/a+2,9.7/a+1,11.9/a} {
    \draw[nipl,fill=gray!12] (\x,3.55) rectangle ++(2.2,.6);
    \node at (\x+1.1,3.85) {$r=\lab$};
  }
  \foreach \x/\lab in {7.5/0,9.7/1,11.9/2,14.1/3}
    \node at (\x,4.4) {$\lab$};
  \node at (10.8,3.1) {parameter activation: fine to coarse};
  \node (pa) at (7.95,2.05) {$P_a$};
  \node (pa1) at (9.85,2.05) {$P_{a+1}$};
  \node (pa2) at (11.85,2.05) {$P_{a+2}$};
  \node (pb) at (13.7,2.05) {$P_b$};
  \draw[nipa] (pa) -- (pa1) node[midway,above=4pt] {$L_{a,1}$};
  \draw[nipa] (pa1) -- (pa2) node[midway,above=4pt] {$L_{a+1,1}$};
  \draw[nipa] (pa2) -- (pb) node[midway,above=4pt] {$L_{a+2,1}$};
  \node at (10.8,1.35) {point application: coarse to fine};
  \node at (7.15,.25)
    {$G_{a,b}(t)^{-1}G_{a,b}(s)=L_{r,s-t}$ \quad ($s,t$ in one scale-$r$ phase)};
\end{tikzpicture}
\caption{Protected squeezing and phase order. The cores satisfy
$L_{r,1}(P_r)=P_{r+1}$. The outer intervals contain the supports at successive
scales; $L_{r,1}$ preserves $J_r$. Parameter activation proceeds from fine to coarse
scales, whereas the full composite acts on a point from coarse to fine
scales. In one scale-$r$ phase the completed fine-scale factors cancel,
leaving $L_{r,s-t}$. Interval widths are schematic.}
\label{fig:protected-phase-order}
\end{figure}
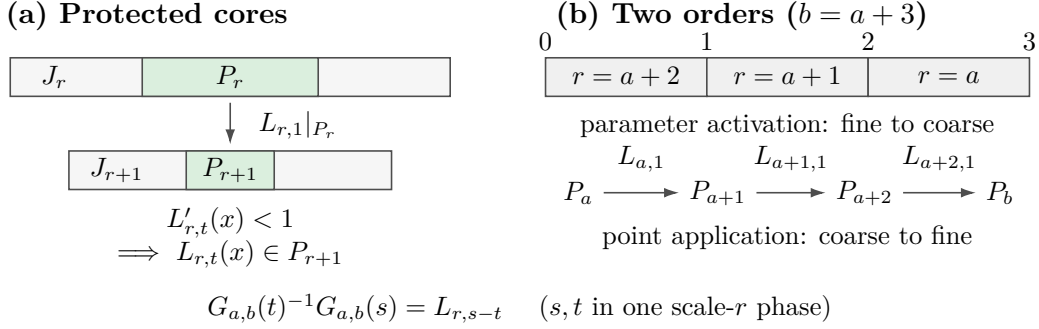

\section{Batched rotation towers and the protected moving front}\label{sec:towers}

We use consecutive blocks of micro-scales
\[
 K=a_0<b_0=a_1<b_1=a_2<\cdots,
 \qquad w_m:=b_m-a_m,
\]
and a strictly increasing sequence of positive integer tower heights
$T_m$.  Put $T_{-1}=0$,
$N_m=T_{m-1}$, and $D_m=T_m-N_m$.  We require
\begin{equation}\label{eq:tower-separation}
 D_m\ge2w_m,\qquad
 \min_{1\le n\le2T_m+4}\normT{n\alpha}>2A_{a_m}\delta_{a_m}.
\end{equation}
The schedules in Section~\ref{sec:arithmetic-schedules} have these
properties.  In particular, the intervals
$R_\alpha^jJ_{a_m}$, $|j|\le T_m+2$, have pairwise disjoint closures.

Fix an allocation exponent $0\le\eta<1$.  It will be $0$ for the $C^1$
construction and $\beta$ for the $C^{1,\beta}$ construction.  Set
\[
 U_m(\eta):=\sum_{r=a_m}^{b_m-1}\delta_r^{-\eta}.
\]
First define
\[
 d_{m,r}^{0}:=
 \left\lceil\frac{(D_m-w_m)\delta_r^{-\eta}}{U_m(\eta)}\right\rceil.
\]
Their sum is at most $D_m$.  We add the unused integer levels to
$d_{m,a_m}^{0}$ and denote the resulting positive integers by $d_{m,r}$;
thus
\begin{equation}\label{eq:allocation-bounds}
 \sum_{r=a_m}^{b_m-1}d_{m,r}=D_m,\qquad
 \frac{\delta_r^{-\eta}}{d_{m,r}}
 \le\frac{U_m(\eta)}{D_m-w_m}
 \le\frac{2U_m(\eta)}{D_m}.
\end{equation}

Let $e_{m,b_m}=0$ and
$e_{m,r}=\sum_{s=r}^{b_m-1}d_{m,s}$.  We define a continuous nondecreasing
function $\varphi_m:\R\to[0,w_m]$ by setting it equal to $0$ on
$(-\infty,0]$, to $w_m$ on $[D_m,\infty)$, and, for
$e_{m,r+1}\le x\le e_{m,r}$,
\begin{equation}\label{eq:staircase-profile}
 \varphi_m(x)=b_m-1-r+\frac{x-e_{m,r+1}}{d_{m,r}}.
\end{equation}
Finally put
\begin{equation}\label{eq:tower-phase}
 \sigma_{m,j}:=\varphi_m(T_m+1-|j|),\qquad j\in\Z.
\end{equation}
By \eqref{eq:tower-phase}, $\sigma_{m,j}=w_m$ for $|j|\le N_m+1$ and
$\sigma_{m,j}=0$ for $|j|\ge T_m+1$.  Because every breakpoint
$e_{m,r}$ is an integer, each active adjacent pair lies in one phase, and
for a unique $r$ its parameter difference has absolute value $d_{m,r}^{-1}$.

\begin{definition}\label{def:tower-regions}
The block of integer micro-scales $a_m\le r<b_m$ is called
\emph{macro-stage $m$}, with width $w_m=b_m-a_m$.  Its \emph{rotation tower}
is the family
\[
 \mathcal T_m:=\{R_\alpha^jJ_{a_m}:j\in\Z,\ |j|\le T_m+2\};
\]
$R_\alpha^jJ_{a_m}$ is its level $j$.  The \emph{old plateau} consists of
the levels with $|j|\le N_m+1$, equivalently $\sigma_{m,j}=w_m$.
The \emph{taper levels} are those with
$N_m+1<|j|<T_m+1$, equivalently $0<\sigma_{m,j}<w_m$.
The remaining tower levels satisfy $\sigma_{m,j}=0$.
These equivalences refer to the staircase profile
\eqref{eq:staircase-profile}--\eqref{eq:tower-phase}.

A \emph{tower edge} is the ordered pair of levels $(j,j+1)$.
For $-T_m-1\le j\le T_m$, its \emph{source} and \emph{target intervals} are
\[
 J_{m,j}^{\mathrm{src}}:=R_\alpha^jJ_{a_m},\qquad
 J_{m,j}^{\mathrm{tgt}}:=R_\alpha^{j+1}J_{a_m}
 =R_\alpha(J_{m,j}^{\mathrm{src}}).
\]
It is \emph{active} when $\sigma_{m,j}\ne\sigma_{m,j+1}$.
For the chosen staircase profile, the active index set is
\[
 \mathcal I_m=\mathcal I_m^-\sqcup\mathcal I_m^+,\qquad
 \mathcal I_m^-:=\{-T_m-1,\ldots,-N_m-2\},\quad
 \mathcal I_m^+:=\{N_m+1,\ldots,T_m\}.
\]
The pairs with indices in $\mathcal I_m^-$ and $\mathcal I_m^+$ form the
\emph{negative} and \emph{positive active fronts}, respectively.
We call them the \emph{fresh tower edges}: there are exactly
$D_m=T_m-N_m$ on each front.  The fronts are blocks of orbit indices;
their source unions need not be connected subsets of the circle.

An active pair is \emph{active at scale $r$} if the open interval between
$\sigma_{m,j}$ and $\sigma_{m,j+1}$ is contained in
$(b_m-1-r,b_m-r)$, the interior of the scale-$r$ phase of
$G_{a_m,b_m}$.  Let $\mathcal I_{m,r}$ be its set of indices.  The integer
breakpoints of \eqref{eq:staircase-profile} give
$\mathcal I_m=\bigsqcup_{r=a_m}^{b_m-1}\mathcal I_{m,r}$;
each front has exactly $d_{m,r}$ scale-$r$ edges, and on such an edge
$|\sigma_{m,j+1}-\sigma_{m,j}|=d_{m,r}^{-1}$.
Indeed, the interval $[e_{m,r+1},e_{m,r}]$ has integer endpoints and
length $d_{m,r}$, and the staircase is affine there with slope
$d_{m,r}^{-1}$.  We call this phase-by-phase subdivision a
\emph{batched front}.

For an active pair, $J_{m,j}^{\mathrm{src}}$ is its
\emph{active source interval} and $J_{m,j}^{\mathrm{tgt}}$ its
\emph{corresponding active target interval}.  Their unions are
\[
 E_m^{\mathrm{src}}=\bigcup_{j\in\mathcal I_m}J_{m,j}^{\mathrm{src}},
 \qquad
 E_m^{\mathrm{tgt}}=\bigcup_{j\in\mathcal I_m}J_{m,j}^{\mathrm{tgt}}
 =R_\alpha(E_m^{\mathrm{src}}).
\]
Activity concerns a pair, whereas the plateau and taper are collections
of levels.  A plateau endpoint or a zero-phase level may belong to an
active pair; an active target need not be an active source for the next
pair.  The protected block \eqref{eq:protected-region} uses only the
cores with $|j|\le N_m$ and is distinct from the old plateau.
\end{definition}

Define $M_m$ levelwise by
\begin{equation}\label{eq:Mm}
 M_m|_{R_\alpha^jJ_{a_m}}
 =R_\alpha^jG_{a_m,b_m}(\sigma_{m,j})R_\alpha^{-j},
 \qquad |j|\le T_m+2,
\end{equation}
and by the identity elsewhere.  Nested supports and stationarity near their
endpoints make the pieces smooth.  We set
\begin{equation}\label{eq:hmqm}
 h_0=\id,\qquad h_{m+1}=M_mh_m,\qquad
 g_m=h_m^{-1}R_\alpha h_m.
\end{equation}

\Needspace{18\baselineskip}
Figure~\ref{fig:batched-active-edges} records the distinction between levels and active edges.

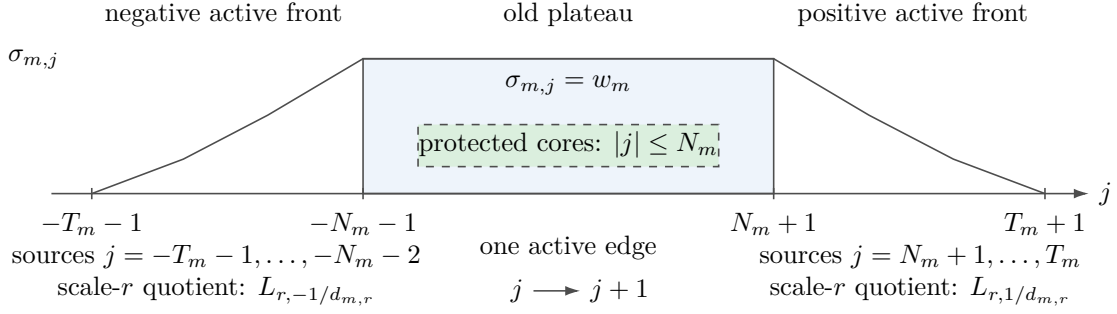
\begin{figure}[H]
\centering
\begin{tikzpicture}[x=.97cm,y=.83cm,font=\small,>=Latex]
  \node at (-4.7,2.85) {negative active front};
  \node at (0,2.85) {old plateau};
  \node at (4.7,2.85) {positive active front};
  \draw[nipl,fill=softblue] (-2.8,0) rectangle (2.8,2.15);
  \draw[nipl] (-6.5,0) -- (-5.25,.55) -- (-4.1,1.25)
    -- (-2.8,2.15) -- (2.8,2.15) -- (4.1,1.25)
    -- (5.25,.55) -- (6.5,0);
  \node[anchor=east] at (-6.8,2.15) {$\sigma_{m,j}$};
  \node at (0,1.73) {$\sigma_{m,j}=w_m$};
  \draw[nipl,dashed,fill=coregreen] (-2.05,.43) rectangle (2.05,1.1);
  \node at (0,.78) {protected cores: $|j|\le N_m$};
  \draw[nipa] (-7.05,0) -- (7.1,0) node[right] {$j$};
  \foreach \x in {-6.5,-2.8,2.8,6.5}
    \draw[nipl] (\x,.09) -- (\x,-.09);
  \node[anchor=north] at (-6.5,-.15) {$-T_m-1$};
  \node[anchor=north] at (-2.8,-.15) {$-N_m-1$};
  \node[anchor=north] at (2.8,-.15) {$N_m+1$};
  \node[anchor=north] at (6.5,-.15) {$T_m+1$};
  \node[align=center] at (-4.8,-1.3)
    {sources $j=-T_m-1,\ldots,-N_m-2$\\
     scale-$r$ quotient: $L_{r,-1/d_{m,r}}$};
  \node[align=center] at (4.8,-1.3)
    {sources $j=N_m+1,\ldots,T_m$\\
     scale-$r$ quotient: $L_{r,1/d_{m,r}}$};
  \node at (0,-.85) {one active edge};
  \node (src) at (-.7,-1.55) {$j$};
  \node (tgt) at (.7,-1.55) {$j+1$};
  \draw[nipa] (src) -- (tgt);
\end{tikzpicture}
\caption{A batched front in orbit-index coordinates. Each front has
$d_{m,r}$ scale-$r$ edges. Activity belongs to an edge $(j,j+1)$; the
plateau endpoints may be an active source or target, while the protected
cores use only $|j|\le N_m$. The middle quotient in
\eqref{eq:moving-front} has negative time on the negative front and
positive time on the positive front. The horizontal ordering is by
orbit index, not by spatial position on the circle. The profile is schematic.}
\label{fig:batched-active-edges}
\end{figure}

\Needspace{14\baselineskip}
\begin{lemma}\label{lem:moving-front}
Let $S_m=M_m^{-1}R_\alpha M_m$.  On a source level with
$-T_m-1\le j\le T_m$,
\begin{equation}\label{eq:moving-front}
 S_m
 =R_\alpha^{j+1}G_{a_m,b_m}(\sigma_{m,j+1})^{-1}
 G_{a_m,b_m}(\sigma_{m,j})R_\alpha^{-j}.
\end{equation}
If $j\in\mathcal I_{m,r}$, the middle quotient is
$L_{r,\pm d_{m,r}^{-1}}$.  The defect $S_m-R_\alpha$ is supported in
$E_m^{\mathrm{src}}$, both maps carry every active source interval onto its
target, and
\begin{align}
 \|S_m-R_\alpha\|_{C^0(J_{m,j}^{\mathrm{src}})}
 &\le C_{\mathrm{sq},0}\frac{\delta_r}{d_{m,r}},\label{eq:circle-front-0}\\
 \|DS_m-1\|_{C^0(J_{m,j}^{\mathrm{src}})}
 &\le C_{\mathrm{sq},1}\frac1{d_{m,r}}.\label{eq:circle-front-1}
\end{align}
\end{lemma}

\begin{proof}
The first map applies the composite attached to level $j$; after rotation,
the inverse composite attached to level $j+1$ is applied.  This proves
\eqref{eq:moving-front}.  The phase alignment in
\eqref{eq:staircase-profile} and Lemma~\ref{lem:batched-squeeze} give the
single-flow quotient and estimates
\eqref{eq:circle-front-0}--\eqref{eq:circle-front-1}.  Equal adjacent parameters
cancel, proving the support assertion.
\end{proof}

\section{Protected density and stabilized wandering intervals}\label{sec:density}

For a circle diffeomorphism $h$, let
$\rho_h(y)=h'(h^{-1}(y))$ be its \emph{rotation-coordinate density}.
Write $\rho_m=\rho_{h_m}$ and define the \emph{protected block} by
\begin{equation}\label{eq:protected-region}
 \mathcal P_m=\bigcup_{|j|\le N_m}R_\alpha^jP_{a_m}.
\end{equation}

\begin{proposition}\label{prop:protected-density}
For every $m$,
\[
 \rho_m(y)\ge1\qquad(y\in\T\setminus\mathcal P_m).
\]
Moreover,
$(E_m^{\mathrm{src}}\cup E_m^{\mathrm{tgt}})\cap\mathcal P_m=\varnothing$.
\end{proposition}

\begin{proof}
The assertion is immediate for $m=0$.  Suppose it holds at stage $m$.
If $y=M_m(z)$, then \eqref{eq:hmqm} gives
\begin{equation}\label{eq:rho-rec}
 \rho_{m+1}(y)=M_m'(z)\rho_m(z).
\end{equation}
If $z\in\mathcal P_m$, the levelwise definition \eqref{eq:Mm} applies the full
composite, which sends
$P_{a_m}$ to $P_{b_m}=P_{a_{m+1}}$; hence
$M_m(z)\in\mathcal P_{m+1}$.
Thus $y\notin\mathcal P_{m+1}$ implies $z\notin\mathcal P_m$ and
$\rho_m(z)\ge1$.
If also $M_m'(z)<1$, then the coefficient on its tower level is positive,
so the level index has absolute value at most $T_m=N_{m+1}$.  The composite
protection \eqref{eq:batched-protection} places $M_m(z)$ in the corresponding
translate of $P_{b_m}$, again inside $\mathcal P_{m+1}$.  This contradiction gives
$M_m'(z)\ge1$, and \eqref{eq:rho-rec} closes the induction.

Finally, an active pair indexed by $j$ satisfies
$|j|\ge N_m+1$ and $|j+1|\ge N_m+1$, by the plateau condition.
A plateau endpoint may participate in an active pair, but neither its
source nor its target has index of absolute value at most $N_m$.
Padded separation therefore makes both levels disjoint from
$\mathcal P_m$ in \eqref{eq:protected-region}, proving the final claim.
\end{proof}

\begin{lemma}\label{lem:stabilization}
For each $j\in\Z$, choose $m$ so large that $|j|\le N_m$ and set
\[
 I_j=h_m^{-1}(R_\alpha^jP_{a_m}).
\]
This interval is independent of the sufficiently large choice of $m$.
The family $(I_j)_{j\in\Z}$ consists of pairwise disjoint nondegenerate
intervals, and $g_m(I_j)=I_{j+1}$ for every sufficiently large $m$.
\end{lemma}

The commuting diagram in Figure~\ref{fig:fixed-pullback-density}
explains the independence of the physical pullback from the stage.

\begin{figure}[!htbp]
\centering
\begin{tikzpicture}[x=.97cm,y=.9cm,font=\small,>=Latex]
  \node[anchor=west,font=\bfseries] at (0,3.4) {(a) Fixed pullback, shrinking core};
  \draw[nipl,fill=coregreen] (.2,2.25) rectangle (2.2,2.55);
  \draw[nipl,fill=coregreen] (.2,.7) rectangle (2.2,1);
  \node at (1.2,2.95) {$I_j$};
  \node at (1.2,.3) {$I_j$};
  \draw[nipa] (1.2,2.1) -- (1.2,1.15) node[midway,left] {$\id$};
  \draw[nipl,fill=coregreen] (4.1,2.25) rectangle (6.1,2.55);
  \draw[nipl,fill=coregreen] (4.6,.7) rectangle (5.6,1);
  \node at (5.1,2.95) {$R_\alpha^jP_{a_m}$};
  \node at (5.1,.3) {$R_\alpha^jP_{a_{m+1}}$};
  \draw[nipa] (2.4,2.4) -- (3.9,2.4) node[midway,above] {$h_m$};
  \draw[nipa] (2.4,.85) -- (4.35,.85) node[midway,above] {$h_{m+1}$};
  \draw[nipa] (5.1,2.1) -- (5.1,1.15) node[midway,right] {$M_m$};
  \node[anchor=west,font=\bfseries] at (7.7,3.4) {(b) Protection of the density};
  \node at (11,2.7) {$\rho_{m+1}(M_mz)=M_m'(z)\rho_m(z)$};
  \node[align=center] at (11,1.75)
    {$M_m'(z)<1$\\$\Longrightarrow\ M_m(z)\in\mathcal P_{m+1}$};
  \node[nipbox,text width=5.4cm] at (11,.6)
    {active source and target regions\\$\rho_m\ge1$};
  \node at (7,-.45) {$h_{m+1}=M_mh_m$, \qquad $|j|\le N_m$};
\end{tikzpicture}
\caption{A shrinking core with a fixed physical pullback. For a protected
orbit index $j$, the full squeeze sends $R_\alpha^jP_{a_m}$ onto
$R_\alpha^jP_{a_{m+1}}$, while $h_{m+1}=M_mh_m$ keeps $I_j$ unchanged.
Contraction images lie in the new protected block; the active source and
target regions consequently retain $\rho_m\ge1$. The cores shrink while
the set of protected orbit indices grows. Interval widths are schematic.}
\label{fig:fixed-pullback-density}
\end{figure}
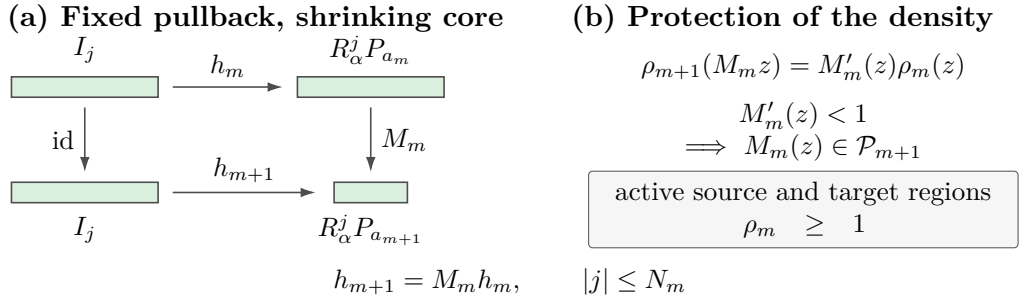

\begin{proof}
On the plateau, $M_m$ uses the full composite, and therefore
\[
 M_m(R_\alpha^jP_{a_m})=R_\alpha^jP_{b_m}
 =R_\alpha^jP_{a_{m+1}}.
\]
Taking inverse images under $h_{m+1}=M_mh_m$ proves stabilization.
Pairwise disjointness and nondegeneracy follow at one common late stage.
Once $j$ and $j+1$ are protected,
$g_m(I_j)=h_m^{-1}R_\alpha(R_\alpha^jP_{a_m})=I_{j+1}$.
\end{proof}

Since $I_0=P_{a_0}$ while $|h_m(I_0)|=|P_{a_m}|\to0$, the construction
genuinely blows up the rotation orbit of $\theta_0$.
\section{Hamiltonian realization in a shrinking boundary collar}\label{sec:hamiltonian}

Let
\[
 \mathbb A=\T\times[0,1],\qquad \omega=dx\wedge ds,
\]
where $s=0$ is the outer boundary, and write
$\mathcal R_\alpha(x,s)=(x+\alpha,s)$.  Fix a smooth cutoff
$\chi:[0,\infty)\to[0,1]$ which equals one on $[0,\mathfrak b]$ and
vanishes near $[1,\infty)$, with $\mathfrak b$ as in \eqref{eq:b-choice}.

At the micro-scale $r$, define
\[
 V_{r,\delta_r}(x)
 =\delta_rV_r\!\left(\frac{x-\theta_0}{\delta_r}\right)
\]
on $J_r$, with zero extension outside, and set
\begin{equation}\label{eq:scaled-Hamiltonian}
 K_r(x,s)
 =sV_{r,\delta_r}(x)\chi(s/\delta_r)
 =\delta_r^2\widehat K_r\!\left(
 \frac{x-\theta_0}{\delta_r},\frac{s}{\delta_r}
 \right),
\end{equation}
where $\widehat K_r(z,w)=wV_r(z)\chi(w)$.  We use the Hamiltonian convention
\[
 X_H=(\partial_sH,-\partial_xH)
\]
for $dx\wedge ds$, and denote the flow by $\Lambda_{r,t}$.

The one-dimensional squeeze has the following area-preserving Hamiltonian realization.

\begin{lemma}\label{lem:ham-squeeze}
The flow $\Lambda_{r,t}$ is a smooth area-preserving diffeomorphism supported
in $J_r\times[0,\delta_r]$, and
\begin{equation}\label{eq:ham-boundary}
 \Lambda_{r,t}(x,0)=(L_{r,t}(x),0).
\end{equation}
On every orbit segment contained in the \emph{exact cotangent zone}
$0\le s\le\mathfrak b\delta_r$,
\begin{equation}\label{eq:cotangent-local}
 \Lambda_{r,t}(x,s)
 =\left(L_{r,t}(x),\frac{s}{L_{r,t}'(x)}\right).
\end{equation}
Moreover, there are constants $C_{\mathrm{ham},\ell}>0$, $0\le\ell\le3$,
depending only on the fixed one-dimensional squeeze
data chosen in Lemma \ref{lem:protected-squeeze} and on $\chi$, such that,
uniformly in the micro-scale $r$ and for $|u|\le1$,
\begin{align}
 \|\Lambda_{r,u}-\id\|_{C^0}&\le C_{\mathrm{ham},0}\delta_r|u|,\label{eq:ham-scale-0}\\
 \|D\Lambda_{r,u}-I\|_{C^0}&\le C_{\mathrm{ham},1}|u|,\label{eq:ham-scale-1}\\
 \|D^2\Lambda_{r,u}\|_{C^0}&\le C_{\mathrm{ham},2}\delta_r^{-1}|u|,\label{eq:ham-scale-2}\\
 \|D^3\Lambda_{r,u}\|_{C^0}&\le C_{\mathrm{ham},3}\delta_r^{-2}|u|.\label{eq:ham-scale-3}
\end{align}
Finally, for $0\le u\le t\le1$,
\begin{equation}\label{eq:derivative-ratio-variable}
 \frac{L_{r,t}'(x)}{L_{r,u}'(x)}\le\lambda_r.
\end{equation}
\end{lemma}

\begin{proof}
The support and symplectic assertions follow from \eqref{eq:scaled-Hamiltonian}.  On $s=0$ the Hamilton equations are $\dot x=V_{r,\delta_r}(x)$ and $\dot s=0$, proving \eqref{eq:ham-boundary}.  Where $\chi\equiv1$ they become
\[
 \dot x=V_{r,\delta_r}(x),\qquad
 \dot s=-sV_{r,\delta_r}'(x).
\]
Thus $x(t)=L_{r,t}(x)$, and the one-dimensional variational equation gives
$s(t)L_{r,t}'(x)=s(0)$, which is \eqref{eq:cotangent-local}.

By Lemma~\ref{lem:protected-squeeze}, the functions $V_r$ have uniform
fixed-order bounds.  The standard variational equations therefore give
uniform bounds through third spatial order.  Conjugation by
$(z,w)\mapsto(\theta_0+\delta_rz,\delta_rw)$ gives
\eqref{eq:ham-scale-0}--\eqref{eq:ham-scale-3}.  The factor $|u|$ follows
from the identity map at time zero.

Finally,
\[
 \frac{L_{r,t}'(x)}{L_{r,u}'(x)}
 =\exp\!\left(\int_u^t
 V_{r,\delta_r}'(L_{r,v}(x))\,dv\right)
 \le e^{\varepsilon_r(t-u)/2}\le\lambda_r.
\]
\end{proof}

For $r\in[a_m,b_m)$ and $j\in\Z$, set
\[
 c_{m,r,j}:=\tau_{a_m,b_m,r}(\sigma_{m,j}).
\]
We define a Hamiltonian diffeomorphism $\mathcal N_{m,r}$ by
\begin{equation}\label{eq:micro-Hamiltonian-factor}
 \mathcal N_{m,r}|_{\mathcal R_\alpha^j(J_r\times[0,\delta_r])}
 =\mathcal R_\alpha^j\Lambda_{r,c_{m,r,j}}\mathcal R_\alpha^{-j},
 \qquad |j|\le T_m+2,
\end{equation}
and by the identity elsewhere.  Indeed, \eqref{eq:micro-Hamiltonian-factor}
is the time-one map of the finite sum
\[
 \sum_{|j|\le T_m+2}c_{m,r,j}\,
 K_r\circ\mathcal R_\alpha^{-j};
\]
the summands have disjoint supports by \eqref{eq:tower-separation}.  We then
set
\begin{equation}\label{eq:batched-Hamiltonian-map}
 \mathcal M_m
 :=\mathcal N_{m,b_m-1}\circ\cdots\circ\mathcal N_{m,a_m}.
\end{equation}
Thus $\mathcal M_m$ is a smooth Hamiltonian diffeomorphism and its boundary
restriction is $M_m$.  This factorization is important: the auxiliary path
obtained by replacing $\sigma_{m,j}$ with a real parameter is only piecewise
smooth at the phase boundaries and is not used as a Hamiltonian isotopy.

\begin{definition}\label{def:active-rectangles}
The strip $\T\times[0,\delta_r]$ is the scale-$r$ \emph{boundary collar},
and $Q_r=J_r\times[0,\delta_r]$ is its local \emph{support rectangle}.
The substrip $\T\times[0,\mathfrak b\delta_r]$, where the cutoff equals
one, is the \emph{exact cotangent zone}; its part inside $Q_r$ is
$J_r\times[0,\mathfrak b\delta_r]$.  Formula
\eqref{eq:cotangent-local} requires the entire orbit segment to remain
in the exact zone.  For $j\in\mathcal I_{m,r}$, the \emph{active source rectangle}
and its \emph{corresponding active target rectangle} are
\[
 Q_{m,r,j}^{\mathrm{src}}=\mathcal R_\alpha^jQ_r,
 \qquad
 Q_{m,r,j}^{\mathrm{tgt}}=\mathcal R_\alpha^{j+1}Q_r.
\]
Their unions over all active pairs are denoted by
$\mathcal E_m^{\mathrm{src}}$ and $\mathcal E_m^{\mathrm{tgt}}$.
These rectangles lie in rotation coordinates.  Their \emph{physical
source and target components} are, respectively,
$\mathcal H_m^{-1}(Q_{m,r,j}^{\mathrm{src}})$ and
$\mathcal H_m^{-1}(Q_{m,r,j}^{\mathrm{tgt}})$, where the cumulative
coordinate change $\mathcal H_m$ is defined in \eqref{eq:2d-HF} below.
\end{definition}

Set
\begin{equation}\label{eq:2d-HF}
 \mathcal H_0=\id,\qquad
 \mathcal H_{m+1}=\mathcal M_m\mathcal H_m,\qquad
 F_m=\mathcal H_m^{-1}\mathcal R_\alpha\mathcal H_m.
\end{equation}
On the outer boundary,
\begin{equation}\label{eq:boundary-stage-identification}
 \mathcal H_m(x,0)=(h_m(x),0),\qquad
 F_m(x,0)=(g_m(x),0).
\end{equation}

For notational use in the moving-front identity, define the composite
\[
 \boldsymbol G_{a,b}(s)
 =\Lambda_{b-1,\tau_{a,b,b-1}(s)}\circ\cdots\circ
 \Lambda_{a,\tau_{a,b,a}(s)}.
\]

\begin{lemma}\label{lem:2d-moving-front}
Let $\mathcal S_m=\mathcal M_m^{-1}\mathcal R_\alpha\mathcal M_m$.
On a source level with $-T_m-1\le j\le T_m$,
\begin{equation}\label{eq:2d-moving-front}
 \mathcal S_m
 =\mathcal R_\alpha^{j+1}
 \boldsymbol G_{a_m,b_m}(\sigma_{m,j+1})^{-1}
 \boldsymbol G_{a_m,b_m}(\sigma_{m,j})
 \mathcal R_\alpha^{-j}.
\end{equation}
For $j\in\mathcal I_{m,r}$ the middle quotient is
$\Lambda_{r,\pm d_{m,r}^{-1}}$.  Both maps carry
$Q_{m,r,j}^{\mathrm{src}}$ onto $Q_{m,r,j}^{\mathrm{tgt}}$, and on this
source rectangle
\begin{align}
 \|\mathcal S_m-\mathcal R_\alpha\|_{C^0}
 &\le C_{\mathrm{ham},0}\frac{\delta_r}{d_{m,r}},\label{eq:2d-defect-0}\\
 \|D\mathcal S_m-I\|_{C^0}
 &\le C_{\mathrm{ham},1}\frac1{d_{m,r}},\label{eq:2d-defect-1}\\
 \|D^2\mathcal S_m\|_{C^0}
 &\le C_{\mathrm{ham},2}\frac{\delta_r^{-1}}{d_{m,r}}.
 \label{eq:2d-defect-2}
\end{align}
The defect is supported in $\mathcal E_m^{\mathrm{src}}$.
\end{lemma}

\begin{proof}
Formula \eqref{eq:2d-moving-front} follows from
\eqref{eq:batched-Hamiltonian-map} exactly
as in Lemma~\ref{lem:moving-front}.  Phase alignment cancels all but one
micro-flow.  The support and image assertions follow from
Lemma~\ref{lem:ham-squeeze}, and its estimates
\eqref{eq:ham-scale-0}--\eqref{eq:ham-scale-2} give
\eqref{eq:2d-defect-0}--\eqref{eq:2d-defect-2}.
\end{proof}

\subsection{Nested cotangent zones}

Although a macro-stage contains several micro-factors, each one is met in
its exact cotangent zone when the factors are traced individually.

\begin{lemma}\label{lem:nested-cotangent}
For every $m\ge0$, on $0\le s\le\delta_{a_m}$,
\begin{equation}\label{eq:Hm-inverse-cotangent}
 \mathcal H_m^{-1}(y,s)
 =\left(h_m^{-1}(y),s\rho_m(y)\right),
 \qquad \rho_m(y)=h_m'(h_m^{-1}(y)).
\end{equation}
Equivalently, on the preimage of this strip,
\begin{equation}\label{eq:Hm-cotangent}
 \mathcal H_m(x,s)
 =\left(h_m(x),\frac{s}{h_m'(x)}\right).
\end{equation}
\end{lemma}

Figure~\ref{fig:nested-exact-zones} shows the backward height bounds
used to keep every prior factor in its exact cotangent zone.

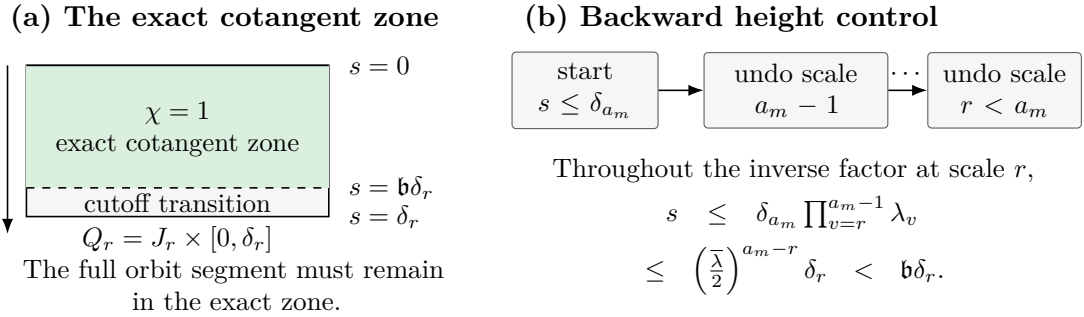
\begin{figure}[!htbp]
\centering
\begin{tikzpicture}[x=1cm,y=1cm,>=Latex,font=\small]
  \node[anchor=west,font=\bfseries] at (0,4.12)
    {(a) The exact cotangent zone};
  \draw[line width=.6pt,fill=gray!7] (.35,1.50) rectangle (4.35,3.50);
  \fill[coregreen] (.35,1.88) rectangle (4.35,3.50);
  \draw[line width=.6pt] (.35,3.50) -- (4.35,3.50);
  \draw[line width=.6pt,dashed] (.35,1.88) -- (4.35,1.88);
  \draw[-{Latex[length=2mm]},line width=.6pt] (.10,3.50) -- (.10,1.27);
  \node[anchor=west] at (4.48,3.50) {$s=0$};
  \node[anchor=west] at (4.48,1.88) {$s=\mathfrak b\delta_r$};
  \node[anchor=west] at (4.48,1.50) {$s=\delta_r$};
  \node[align=center] at (2.35,2.66)
    {$\chi=1$\\exact cotangent zone};
  \node at (2.35,1.68) {cutoff transition};
  \node at (2.35,1.20) {$Q_r=J_r\times[0,\delta_r]$};
  \node[align=center,text width=6.1cm] at (3.10,.60)
    {The full orbit segment must remain\\in the exact zone.};

  \node[anchor=west,font=\bfseries] at (6.80,4.12)
    {(b) Backward height control};
  \node[draw=black!65,fill=gray!7,rounded corners=1.5pt,
    inner sep=4pt,align=center,text width=1.65cm] (start) at (7.74,3.17)
    {start\\$s\le\delta_{a_m}$};
  \node[draw=black!65,fill=gray!7,rounded corners=1.5pt,
    inner sep=4pt,align=center,text width=2.15cm] (first) at (10.53,3.17)
    {undo scale\\$a_m-1$};
  \node[draw=black!65,fill=gray!7,rounded corners=1.5pt,
    inner sep=4pt,align=center,text width=1.75cm] (last) at (13.29,3.17)
    {undo scale\\$r<a_m$};
  \draw[-{Latex[length=2mm]},line width=.6pt] (start.east) -- (first.west);
  \draw[-{Latex[length=2mm]},line width=.6pt] (first.east) -- (last.west)
    node[midway,above] {$\cdots$};
  \node[anchor=north,align=center,text width=7.4cm] at (10.48,2.40)
    {Throughout the inverse factor at scale $r$,\\[4pt]
     $s\le\delta_{a_m}\prod_{v=r}^{a_m-1}\lambda_v$\\[4pt]
     $\le
       \left(\frac{\overline\lambda}{2}\right)^{a_m-r}\delta_r
       <\mathfrak b\delta_r$.};
\end{tikzpicture}
\caption{Nested exact cotangent zones. The backward chain undoes the
prior micro-factors in decreasing scale order. The displayed heights are
upper bounds, not trajectories. They exclude a first exit from each exact
zone, so the cotangent formulas compose along the whole relevant orbit.
All dimensions are schematic.}
\label{fig:nested-exact-zones}
\end{figure}

\begin{proof}
For $m=0$ both identities are immediate.  For $m\ge1$, trace a point
of height at most $\delta_{a_m}$ backward through every prior
micro-factor, in decreasing order of its scale index.  Immediately before
the factor at scale $r<a_m$ is undone, its height is at most
\[
 \delta_{a_m}\prod_{v=r+1}^{a_m-1}\lambda_v
 \le\frac12\left(\frac{\overline\lambda}{2}\right)^{a_m-r-1}
 \delta_r<\mathfrak b\delta_r.
\]
While that inverse factor is traversed, the derivative-ratio estimate
\eqref{eq:derivative-ratio-variable} gives
the strict bound
\[
 \delta_{a_m}\prod_{v=r}^{a_m-1}\lambda_v
 \le\left(\frac{\overline\lambda}{2}\right)^{a_m-r}\delta_r
 \le\frac{\overline\lambda}{2}\delta_r<\mathfrak b\delta_r.
\]
This also excludes a first exit from the exact zone.  Hence every factor is
an exact cotangent lift along the relevant orbit.  Cotangent lifts compose,
which proves \eqref{eq:Hm-inverse-cotangent}; inversion gives
\eqref{eq:Hm-cotangent}.
\end{proof}

\section{The batched Hamiltonian derivative ledger}\label{sec:master-ledger}

The estimates must be propagated through the micro-factors, even though
convergence will be counted one macro-stage at a time.

\subsection{Auxiliary micro-factorization}

Let $N_{m,r}$ be the boundary restriction of $\mathcal N_{m,r}$.  Define
auxiliary circle maps by $\widehat h_K=\id$ and
\begin{equation}\label{eq:auxiliary-micro-conjugacy}
 \widehat h_{r+1}=N_{m,r}\widehat h_r,
 \qquad a_m\le r<b_m.
\end{equation}
The consecutive blocks partition all $r\ge K$, so this recursion is
unambiguous.  Then $h_m=\widehat h_{a_m}$ and
$h_{m+1}=\widehat h_{b_m}$.

\begin{lemma}\label{lem:hk-higher-variable}
Set
\[
 B_{\mathrm{der}}:=\exp\!\left(\frac12\sum_{v=1}^{\infty}v^{-2}\right).
\]
For $\ell=2,3,4$, there are constants $C_{h,\ell}>0$, depending only on the
 fixed one-dimensional squeeze data chosen in Lemma \ref{lem:protected-squeeze}, such that, for every $r\ge K$,
\begin{equation}\label{eq:hk-higher-variable}
 \|\widehat h_r'\|_\infty\le B_{\mathrm{der}},
 \qquad
 \|\widehat h_r^{(\ell)}\|_\infty\le C_{h,\ell}\delta_r^{1-\ell}.
\end{equation}
These constants are independent of the block widths, tower heights, and
every admissible increase of $K$.
\end{lemma}

\begin{proof}
Disjointness of the tower levels and \eqref{eq:L-first} give
$\|N_{m,r}'\|_\infty\le\lambda_r$, hence the first estimate.  Also
$\|N_{m,r}^{(\ell)}\|_\infty\le C_{\mathrm{sq},\ell}\delta_r^{1-\ell}$ by
\eqref{eq:Lhigher}.
Applying Fa\`a di Bruno's formula to
\eqref{eq:auxiliary-micro-conjugacy}, one obtains, successively for
$\ell=2,3,4$,
\[
 \|\widehat h_{r+1}^{(\ell)}\|_\infty
 \le\lambda_r\|\widehat h_r^{(\ell)}\|_\infty
 +C_{\mathrm{rec},\ell}\delta_r^{1-\ell},
\]
where $C_{\mathrm{rec},\ell}$ depends only on the fixed squeeze bounds,
$B_{\mathrm{der}}$, and the already fixed lower-order constants.  Since
$\delta_{r+1}^{1-\ell}=2^{\ell-1}\delta_r^{1-\ell}$ and
$2^{\ell-1}>\overline\lambda$, choosing
$C_{h,\ell}\ge C_{\mathrm{rec},\ell}/(2^{\ell-1}-\overline\lambda)$ closes the
induction.  The initial higher derivatives vanish because
$\widehat h_K=\id$, so these choices also cover the initial scale.
\end{proof}

Put
$\mathcal A_m=\mathcal E_m^{\mathrm{src}}\cup
\mathcal E_m^{\mathrm{tgt}}$ and
$\mathcal U_m=\mathcal H_m^{-1}(\mathcal A_m)$.  By
Proposition~\ref{prop:protected-density},
\begin{equation}\label{eq:active-rho-lower}
 \rho_m\ge1\quad\text{on the angular projection of }\mathcal A_m.
\end{equation}
On $\mathcal U_m$, Lemma~\ref{lem:nested-cotangent} identifies
$\mathcal H_m$ with
\begin{equation}\label{eq:Cm}
 C_m(x,s)=\left(h_m(x),\frac{s}{h_m'(x)}\right).
\end{equation}

\begin{lemma}\label{lem:cotangent-derivatives-variable}
There are constants $C_{\mathrm{cot},\ell}>0$, $\ell=1,2,3$, depending only on the fixed one-dimensional squeeze
data chosen in Lemma \ref{lem:protected-squeeze}, such that
\begin{align}
 \|DC_m^{\pm1}\|&\le C_{\mathrm{cot},1},\label{eq:C1-variable}\\
 \|D^2C_m^{\pm1}\|&\le C_{\mathrm{cot},2}\delta_{a_m}^{-1},\label{eq:C2-variable}\\
 \|D^3C_m^{\pm1}\|&\le C_{\mathrm{cot},3}\delta_{a_m}^{-2}.
 \label{eq:C3-variable}
\end{align}
The bounds hold on the respective active domains and, for the inverse,
along each segment in a lifted active target rectangle.  They are uniform
in $m$, the tower data, and every admissible increase of $K$.
\end{lemma}

\begin{proof}
Apply \eqref{eq:hk-higher-variable} at $r=a_m$.  If
$p=h_m'$ and $C_m(x,s)=(y,v)$, then $p(x)=\rho_m(y)$; hence on the active
regions, by \eqref{eq:active-rho-lower},
\[
 1\le p\le B_{\mathrm{der}},\qquad
 |p^{(\ell)}|\le C_{h,\ell+1}\delta_{a_m}^{-\ell},
 \quad \ell=1,2,3.
\]
Moreover, $v\le\delta_{a_m}$ and $s=vp\le
B_{\mathrm{der}}\delta_{a_m}$.  Differentiating the formula \eqref{eq:Cm} gives
\eqref{eq:C1-variable}--\eqref{eq:C3-variable}.  For the inverse use
\[
 C_m^{-1}(y,v)=\bigl(h_m^{-1}(y),v\rho_m(y)\bigr)
\]
and differentiate; every reciprocal of $p$ is bounded by one, and the
factor $v\le\delta_{a_m}$ removes one power of
$\delta_{a_m}^{-1}$ in the second component.  A lifted target rectangle is
convex, so the same estimates hold along its joining segments.
\end{proof}

\subsection{The macro-stage increment}

Set $\Delta_m=F_{m+1}-F_m$ in consistent annulus lifts.  From
\eqref{eq:2d-HF}, on its physical support,
\begin{equation}\label{eq:conjugation-difference-variable}
 F_{m+1}=C_m^{-1}\mathcal S_mC_m,\qquad
 F_m=C_m^{-1}\mathcal R_\alpha C_m.
\end{equation}

\begin{proposition}\label{prop:master-ledger}
There are constants $C_{\Delta,\ell}>0$, $\ell=0,1,2$, depending only on the
 fixed one-dimensional squeeze
data chosen in Lemma \ref{lem:protected-squeeze} and on $\chi$, with the following property.  On the
physical preimage of an active scale-$r$ source rectangle,
\begin{align}
 \|\Delta_m\|_{C^0}&\le
 C_{\Delta,0}\frac{\delta_r}{d_{m,r}},\label{eq:Delta0-master}\\
 \|D\Delta_m\|_{C^0}&\le
 C_{\Delta,1}\frac1{d_{m,r}},\label{eq:Delta1-master}\\
 \|D^2\Delta_m\|_{C^0}&\le
 C_{\Delta,2}\frac{\delta_r^{-1}}{d_{m,r}}.
 \label{eq:Delta2-master}
\end{align}
The constants are independent of $m,r$, all block and tower data, and every
admissible increase of $K$.
\end{proposition}

\begin{proof}
	Fix an active scale-$r$ component and write
	$\Psi_m=C_m^{-1}$, $w=C_m(z)$,
	$w_{\mathcal S}=\mathcal S_m(w)$, and
	$w_{\mathcal R}=\mathcal R_\alpha(w)$.  Thus, by
	\eqref{eq:conjugation-difference-variable},
	\[
	\Delta_m(z)
	=
	\Psi_m(w_{\mathcal S})-\Psi_m(w_{\mathcal R}).
	\]
	Both $w_{\mathcal S}$ and $w_{\mathcal R}$ lie in the same lifted active
	target rectangle.  Since this rectangle is convex, the inverse cotangent
	bounds \eqref{eq:C1-variable}--\eqref{eq:C3-variable} apply along the
	segment joining them.
	
	By the mean-value theorem, \eqref{eq:C1-variable}, and
	\eqref{eq:2d-defect-0},
	\[
	|\Delta_m(z)|
	\le
	\sup_{[w_{\mathcal R},w_{\mathcal S}]}\|D\Psi_m\|\,
	|w_{\mathcal S}-w_{\mathcal R}|
	\le
	C\frac{\delta_r}{d_{m,r}},
	\]
	which gives \eqref{eq:Delta0-master}.
	
	Differentiating \eqref{eq:conjugation-difference-variable} once gives
	\[
	[D\Psi_m(w_{\mathcal S})-D\Psi_m(w_{\mathcal R})]
	D\mathcal S_mDC_m
	+
	D\Psi_m(w_{\mathcal R})(D\mathcal S_m-I)DC_m.
	\]
	By the mean-value theorem, \eqref{eq:C2-variable}, and
	\eqref{eq:2d-defect-0},
	\[
	\|D\Psi_m(w_{\mathcal S})-D\Psi_m(w_{\mathcal R})\|
	\le
	C\delta_{a_m}^{-1}
	\frac{\delta_r}{d_{m,r}}.
	\]
	Together with \eqref{eq:2d-defect-1} and
	\eqref{eq:C1-variable}, and using
	$\delta_r\le\delta_{a_m}$, this proves
	\eqref{eq:Delta1-master}.
	
	For the second derivative, recall the chain-rule identity
	\[
	D^2(\Psi\circ S\circ C)
	=
	D^2\Psi[DS\,DC,DS\,DC]
	+
	D\Psi\,D^2S[DC,DC]
	+
	D\Psi\,DS\,D^2C.
	\]
	Put
	\[
	P=DC_m(z),\qquad
	Q=D^2C_m(z),\qquad
	S=D\mathcal S_m(w).
	\]
	Since $D\mathcal R_\alpha=I$ and
	$D^2\mathcal R_\alpha=0$, subtracting the corresponding expression for
	$\Psi_m\circ\mathcal R_\alpha\circ C_m$ gives the five groups
	\begin{align*}
		&[D^2\Psi_m(w_{\mathcal S})-D^2\Psi_m(w_{\mathcal R})][SP,SP],\\
		&D^2\Psi_m(w_{\mathcal R})[(S-I)P,SP]
		+D^2\Psi_m(w_{\mathcal R})[P,(S-I)P],\\
		&D\Psi_m(w_{\mathcal S})D^2\mathcal S_m(w)[P,P],\\
		&[D\Psi_m(w_{\mathcal S})-D\Psi_m(w_{\mathcal R})]SQ,\\
		&D\Psi_m(w_{\mathcal R})(S-I)Q.
	\end{align*}
	Here the second line is the expansion of one bilinear-factor difference.
	The segment between $w_{\mathcal S}$ and $w_{\mathcal R}$ stays in the
	same lifted active target rectangle, so
	\eqref{eq:C1-variable}--\eqref{eq:C3-variable} apply along it.
	Also, by \eqref{eq:2d-defect-1} and $d_{m,r}\ge1$,
	\[
	\|S\|\le 1+C_{\mathrm{ham},1}.
	\]
	
	Using \eqref{eq:2d-defect-0}--\eqref{eq:2d-defect-2} and
	\eqref{eq:C1-variable}--\eqref{eq:C3-variable}, the five groups above
	are bounded by fixed constants times, respectively,
	\[
	\frac{\delta_{a_m}^{-2}\delta_r}{d_{m,r}},\qquad
	\frac{\delta_{a_m}^{-1}}{d_{m,r}},\qquad
	\frac{\delta_r^{-1}}{d_{m,r}},\qquad
	\frac{\delta_{a_m}^{-2}\delta_r}{d_{m,r}},\qquad
	\frac{\delta_{a_m}^{-1}}{d_{m,r}}.
	\]
	Since $r\ge a_m$, we have $\delta_r\le\delta_{a_m}$, and hence
	\[
	\delta_{a_m}^{-1}\le\delta_r^{-1},
	\qquad
	\delta_{a_m}^{-2}\delta_r\le\delta_r^{-1}.
	\]
	Thus all five groups are bounded by a fixed multiple of
	$\delta_r^{-1}/d_{m,r}$, which proves
	\eqref{eq:Delta2-master}.
	
	For each active pair, restrict $\Delta_m$ to its physical source component
	and extend it by zero to the rest of the closed annulus.  The defect is
	identically zero in a relative neighborhood of the lateral and inner
	edges, so this extension is smooth up to the outer boundary and obeys
	the same bounds.  These component extensions have disjoint supports.
\end{proof}

Figure~\ref{fig:active-ledger} illustrates the local segment estimates
and the disjoint-support gluing used below.

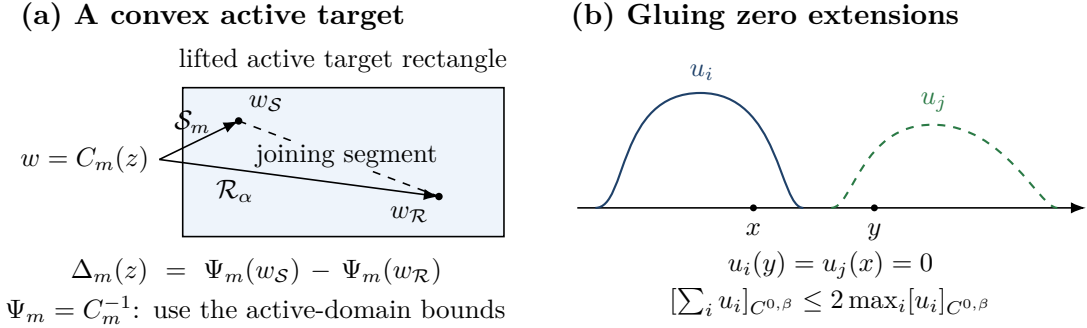
\begin{figure}[!htbp]
\centering
\begin{tikzpicture}[x=1cm,y=1cm,>=Latex,font=\small]
  \node[anchor=west,font=\bfseries] at (0,4.03)
    {(a) A convex active target};
  \begin{scope}[yshift=-3mm]
  \draw[line width=.6pt,fill=softblue] (2.28,1.42) rectangle (6.53,3.39);
  \node[anchor=south] at (4.40,3.46) {lifted active target rectangle};
  \node[anchor=west,align=center] (w) at (0,2.44) {$w=C_m(z)$};
  \coordinate (ws) at (3.02,2.95);
  \coordinate (wr) at (5.67,1.95);
  \draw[-{Latex[length=2mm]},line width=.6pt]
    (w.east) -- (ws) node[pos=.42,above] {$\mathcal S_m$};
  \draw[-{Latex[length=2mm]},line width=.6pt]
    (w.east) -- (wr) node[pos=.27,below] {$\mathcal R_\alpha$};
  \draw[line width=.6pt,dashed] (ws) -- (wr);
  \fill (ws) circle (1.4pt);
  \fill (wr) circle (1.4pt);
  \node[anchor=south west] at (ws) {$w_{\mathcal S}$};
  \node[anchor=north east] at (wr) {$w_{\mathcal R}$};
  \node[fill=softblue,inner sep=1.5pt] at (4.45,2.48) {joining segment};
  \node[align=center,text width=6.7cm] at (3.25,.70)
    {$\Delta_m(z)=\Psi_m(w_{\mathcal S})-\Psi_m(w_{\mathcal R})$\\[3pt]
     $\Psi_m=C_m^{-1}$: use the active-domain bounds};
  \end{scope}

  \node[anchor=west,font=\bfseries] at (7.28,4.03)
    {(b) Gluing zero extensions};
  \draw[-{Latex[length=2mm]},line width=.6pt] (7.50,1.50) -- (14.23,1.50);
  \draw[line width=.8pt,deepblue]
    (7.74,1.50) .. controls (8.17,1.50) and (8.02,3.02) .. (9.13,3.02)
    .. controls (10.18,3.02) and (10.17,1.50) .. (10.49,1.50);
  \draw[line width=.8pt,coregreenline,dashed]
    (10.86,1.50) .. controls (11.15,1.50) and (11.25,2.60) .. (12.20,2.60)
    .. controls (13.15,2.60) and (13.61,1.50) .. (13.85,1.50);
  \node[deepblue] at (9.15,3.29) {$u_i$};
  \node[coregreenline] at (12.22,2.90) {$u_j$};
  \fill (9.83,1.50) circle (1.3pt);
  \fill (11.43,1.50) circle (1.3pt);
  \node[anchor=north] at (9.83,1.43) {$x$};
  \node[anchor=north] at (11.43,1.43) {$y$};
  \node at (10.85,.75) {$u_i(y)=u_j(x)=0$};
  \node at (10.85,.25)
    {$[\sum_i u_i]_{C^{0,\beta}}
      \le2\max_i[u_i]_{C^{0,\beta}}$};
\end{tikzpicture}
\caption{Local conjugation and global gluing. Both images lie in the same
convex lifted target rectangle, allowing the inverse cotangent estimates
along their joining segment. The scalar profiles on the right illustrate
components of the zero-extended increment derivatives; their supports
are disjoint. Any two-point comparison involves at most two
components. No uniform lower bound on their separation is needed.}
\label{fig:active-ledger}
\end{figure}

We use the flat product metric on $\mathbb A=\T\times[0,1]$ and its
standard tangent trivialization for annular derivative norms.  The shortest
segment in compatible lifts stays in $\R\times[0,1]$, so a global bound
on the derivative implies the corresponding Lipschitz bound.  Let $u_i$
be finitely many scalar-, vector-, or matrix-valued functions on this
metric space with pairwise disjoint supports, each understood as a global
zero extension.  Then
\begin{equation}\label{eq:holder-zero-gluing}
 \left[\sum_i u_i\right]_{C^{0,\beta}}
 \le 2\max_i[u_i]_{C^{0,\beta}}.
\end{equation}
Indeed, at a point at most one $u_i$ is nonzero.  If the two points lie
in distinct supports indexed by $i$ and $j$, then $u_i(y)=u_j(x)=0$;
applying the two global H\"older bounds gives the factor two.  If both
points lie in the same support, or one lies outside all supports, a
single H\"older bound suffices.  No lower bound on the distance between
supports is required.

\Needspace{11\baselineskip}
\begin{corollary}\label{cor:holder-stage}
For every $0<\beta<1$, there is a constant
$C_{\mathrm{Hol},\beta}>0$, with the same uniformity as in
Proposition~\ref{prop:master-ledger}, such that
\begin{equation}\label{eq:holder-macro-stage}
 [D\Delta_m]_{C^{0,\beta}}
 \le C_{\mathrm{Hol},\beta}
 \max_{a_m\le r<b_m}\frac{\delta_r^{-\beta}}{d_{m,r}}.
\end{equation}
Moreover,
\begin{equation}\label{eq:C1-macro-stage}
 \|\Delta_m\|_{C^1}
 \le C_{\Delta,01}
 \max_{a_m\le r<b_m}\frac1{d_{m,r}},
 \qquad C_{\Delta,01}:=C_{\Delta,0}+C_{\Delta,1}.
\end{equation}
\end{corollary}

\begin{proof}
Let $V_{m,r,j}$ be the component zero extension constructed above and
put $u=DV_{m,r,j}$.  Equations
\eqref{eq:Delta1-master}--\eqref{eq:Delta2-master} give
$\|u\|_\infty\le A:=C_{\Delta,1}/d_{m,r}$ and
$\operatorname{Lip}(u)\le B:=C_{\Delta,2}\delta_r^{-1}/d_{m,r}$.
For any $x,y$ in the annulus,
\[
 \|u(x)-u(y)\|\le\min\{2A,Bd(x,y)\}
 \le (2A)^{1-\beta}B^\beta d(x,y)^\beta.
\]
Thus each global zero extension satisfies
\[
 [DV_{m,r,j}]_{C^{0,\beta}}
 \le 2^{1-\beta}C_{\Delta,1}^{1-\beta}C_{\Delta,2}^{\beta}
       \frac{\delta_r^{-\beta}}{d_{m,r}}.
\]
Apply \eqref{eq:holder-zero-gluing}; for this norm convention we may take
$C_{\mathrm{Hol},\beta}=2^{2-\beta}
 C_{\Delta,1}^{1-\beta}C_{\Delta,2}^{\beta}$.
The two supremum bounds in the $C^1$ norm are controlled separately by
the corresponding maxima because the component supports are disjoint.
Equation~\eqref{eq:C1-macro-stage} follows from
\eqref{eq:Delta0-master}--\eqref{eq:Delta1-master} and $\delta_r\le1$.
\end{proof}

\begin{remark}\label{rem:ledger-uniformity}
The passage through auxiliary micro-factors is essential.  A whole block has
uniform $C^1$ cost, but its second derivative may see the finest scale
$\delta_{b_m-1}^{-1}$; no coarse-scale $C^2$ estimate is asserted.
\end{remark}
\section{Master convergence and completion to the disk}\label{sec:master-convergence}

We now isolate the abstract summability statement supplied by the derivative ledger.

\begin{proposition}\label{prop:master-convergence}
Fix the block and tower data satisfying
\eqref{eq:tower-separation}, together with the integer allocation
$d_{m,r}$ chosen in Section~\ref{sec:towers}.
Keeping this allocation fixed, for an independent exponent
$0\le t<1$ set
\begin{equation}\label{eq:Xi-definition}
	\Xi_m(t):=
	\max_{a_m\le r<b_m}\frac{\delta_r^{-t}}{d_{m,r}}.
\end{equation}
\begin{enumerate}[label=\textnormal{(\roman*)}]
\item If
\begin{equation}\label{eq:C1-master-condition}
 \sum_m\Xi_m(0)<\infty,
\end{equation}
then $F_m$ converges in $C^1$ to an area-preserving annulus diffeomorphism
$F$.  Its outer boundary restriction is a genuine $C^1$ Denjoy
diffeomorphism of rotation number $\alpha$.
\item If, for some $0<\beta<1$,
\begin{equation}\label{eq:Holder-master-condition}
 \sum_m\Xi_m(\beta)<\infty,
\end{equation}
then $F$ and $F^{-1}$ are $C^{1,\beta}$.
\end{enumerate}
After attaching a rigid central disk, we obtain an area-preserving irrational pseudo-rotation with exactly one periodic point.  The resulting disk map is not $C^0$-rigid, and on $\operatorname{int}(\D)\setminus\{0\}$ it is smoothly conjugate to the rigid rotation.
\end{proposition}

\begin{proof}
	We first assume \eqref{eq:C1-master-condition}.  Estimate
	\eqref{eq:C1-macro-stage} implies
	\[
	\sum_m\|F_{m+1}-F_m\|_{C^1}<\infty,
	\]
	so $F_m\to F$ in $C^1$.
	
Recall that $\mathbb A=\mathbb T\times[0,1]$ carries the coordinates
$(x,s)$ and the area form $dx\wedge ds$. We first verify that the $C^1$ limit is still a global area-preserving diffeomorphism of the closed annulus.  Each $F_m$ is
conjugate to $\mathcal R_\alpha$ by a Hamiltonian diffeomorphism and
satisfies
\[
\det DF_m\equiv1.
\]
Passing to the $C^1$ limit gives $\det DF\equiv1$, hence
\[
F^*(dx\wedge ds)=dx\wedge ds,
\]
and $DF$ is invertible everywhere.  Uniform convergence shows that
$F$ preserves the two boundary circles.  Using the compatible
degree-one lifts of the finite-stage maps, the restriction of $F$ to
each boundary component has degree one.

No interior point can map to the boundary.  Indeed, if
$z\in\operatorname{int}\mathbb A$ and $F(z)\in\partial\mathbb A$, then
the normal coordinate $s\circ F$ attains an interior minimum or maximum
at $z$.  Hence
\[
d(s\circ F)_z=0.
\]
But
\[
d(s\circ F)_z=ds_{F(z)}\circ DF_z,
\]
which is impossible because $ds_{F(z)}\ne0$ and $DF_z$ is invertible.
Thus
\[
F^{-1}(\partial\mathbb A)=\partial\mathbb A.
\]

At a boundary point, the boundary defining function $s\circ F$ is
constant along the boundary component, so its derivative in the
tangential direction vanishes. Since $\det DF=1$, its derivative in
the transverse direction is nonzero. Moreover, because $F$ maps the
annulus into itself and preserves the corresponding boundary
component, this transverse derivative points to the inward side.
Hence the local inverse preserves the inward half-space, and the
inverse function theorem in boundary charts shows that $F$ is a local
$C^1$ diffeomorphism of the closed annulus.

Since $\mathbb A$ is compact, $F$ is proper.  Hence $F$ is a finite
covering map.  Its restriction to each boundary component has degree
one, so the covering degree is one.  Therefore $F$ is a global $C^1$
diffeomorphism.
	
	We next identify the outer boundary dynamics.  On $s=0$,
	\eqref{eq:boundary-stage-identification} shows that
	\[
	g_m\to g:=F|_{s=0}
	\]
	in $C^1$.  Since each $g_m$ is conjugate to $R_\alpha$, continuity of
	the circle rotation number gives
	\[
	\rot(g)=\alpha.
	\]
	For each fixed $j$, Lemma~\ref{lem:stabilization} gives
	$g_m(I_j)=I_{j+1}$ for all large $m$.  Passing to the uniform limit
	yields
	\[
	g(I_j)=I_{j+1}.
	\]
	The intervals $I_j$ are pairwise disjoint and nondegenerate, so their
	interiors form a wandering orbit.  Thus $g$ is a genuine $C^1$ Denjoy
	diffeomorphism.
	
	Suppose now that \eqref{eq:Holder-master-condition} holds.  Since
	$\Xi_m(0)\le\Xi_m(\beta)$, condition
	\eqref{eq:C1-master-condition} also holds, and therefore
	part~\textnormal{(i)} applies.  Combining
	\eqref{eq:C1-macro-stage} with \eqref{eq:holder-macro-stage} gives
	\[
	\sum_m\|F_{m+1}-F_m\|_{C^{1,\beta}}<\infty.
	\]
	Hence $F\in C^{1,\beta}$.  Since $F$ is a $C^1$ diffeomorphism on a
	compact set, $F^{-1}$ is Lipschitz.  Therefore
	$DF\circ F^{-1}$ is still $\beta$-H\"older.  Using
	\[
	DF^{-1}(y)=\bigl(DF(F^{-1}y)\bigr)^{-1},
	\]
	together with the Lipschitz continuity of matrix inversion on the
	compact set
	\[
	DF(\mathbb A)\subset GL(2,\R),
	\]
	we conclude that $DF^{-1}$ is also $\beta$-H\"older.  Thus
	$F^{-1}\in C^{1,\beta}$.
	
	It remains to identify the dynamics on $\T\times(0,1]$.  The important
	point is that the coordinate changes need not converge globally near
	the outer boundary; instead, their change regions shrink to that
	boundary.  Consequently both the coordinate changes and their inverses
	stabilize exactly, rather than merely converge, on every compact subset
	of the open annulus.
	
	Write
	\[
	\mathcal C_m=\T\times[0,\delta_{a_m}]
	\]
	for this paragraph only.  Since $\mathcal M_m$ is the identity off
	$\mathcal C_m$, the two difference sets satisfy
	\[
	\{z:\mathcal H_{m+1}(z)\ne\mathcal H_m(z)\}
	\subset \mathcal H_m^{-1}(\mathcal C_m)
	\subset\T\times[0,B_{\mathrm{der}}\delta_{a_m}],
	\]
	and
	\[
	\{z:\mathcal H_{m+1}^{-1}(z)\ne\mathcal H_m^{-1}(z)\}
	\subset \mathcal C_m.
	\]
	The second inclusion in the first line follows from
	\eqref{eq:Hm-inverse-cotangent} and
	Lemma~\ref{lem:hk-higher-variable}; the second line follows directly
	from
	\[
	\mathcal H_{m+1}^{-1}
	=
	\mathcal H_m^{-1}\mathcal M_m^{-1}.
	\]
	Because $a_m\to\infty$, both sequences stabilize on every compact
	subset of $\T\times(0,1]$.  Their limits take values in this same set,
	as each stabilized value is the value of a finite-stage diffeomorphism
	there.
	
	For a point and its stabilized image, choose one stage beyond both
	stabilization thresholds.  The finite-stage inverse identities then
	show that the two stabilized limits are mutual inverses.  They therefore
	define a smooth area-preserving diffeomorphism
	$\mathcal H_\infty$ on $\T\times(0,1]$, and
	\begin{equation}\label{eq:interior-conjugacy}
		F=\mathcal H_\infty^{-1}\mathcal R_\alpha\mathcal H_\infty
	\end{equation}
	there.  Hence every interior point lies on a smooth invariant circle,
	and the dynamics on that circle is smoothly conjugate to $R_\alpha$.
	
	We now complete the annulus to a disk. Since $\delta_K<1$ and every Hamiltonian perturbation is supported in
	$s\le\delta_K$, all cumulative coordinate changes are the identity on
	$\T\times(\delta_K,1]$.  Hence every $F_m$, and therefore its limit
	$F$, equals $\mathcal R_\alpha$ on this fixed inner collar.
	
	We attach a central disk carrying the same rigid irrational rotation.
	Near the seam, write $x=\theta/(2\pi)$ and use
	\[
	I=I_*+1-s
	\]
	on the annulus and $I=\pi r^2$ on the central disk.  In these
	coordinates both area forms are $dI\wedge dx$, and on a full collar the
	dynamics is
	\[
	(x,I)\longmapsto(x+\alpha,I).
	\]
	Thus the two pieces glue smoothly to an area-preserving disk
	diffeomorphism.  Passing to a fixed round Euclidean disk changes the
	relevant $C^1$ and $C^{1,\beta}$ perturbation norms by at most a fixed
	multiplicative constant, independent of the later starting index $K$.
	
	Moreover, $\mathcal H_\infty=\id$ on the rigid collar, so it extends by
	the identity to a smooth conjugacy on
	$\operatorname{int}(\D)\setminus\{0\}$.
	
	The central rigid disk has only its center as a periodic point;
	\eqref{eq:interior-conjugacy} gives no periodic point in
	$\T\times(0,1]$; and the boundary map has irrational rotation number.
	Therefore
	\[
	\Per(F)=\{0\}.
	\]
	All invariant circles carry rotation number $\alpha$.
	
	Finally, we prove the failure of $C^0$-rigidity.  The boundary map $g$
	has a wandering interval
	\[
	W=(\xi_-,\xi_+).
	\]
	The intervals $g^n(W)$ are pairwise disjoint, hence their lengths tend
	to zero as $n\to\infty$.  For all sufficiently large $n$, this length
	is less than $1/2$, and therefore equals the circle distance between
	the two endpoints of $g^n(W)$.
	
	If a subsequence $g^{n_j}$ converged uniformly to the identity, then
	\[
	g^{n_j}(\xi_-)\to\xi_-,
	\qquad
	g^{n_j}(\xi_+)\to\xi_+,
	\]
	so
	\[
	d_{\T}(g^{n_j}(\xi_-),g^{n_j}(\xi_+))
	\longrightarrow d_{\T}(\xi_-,\xi_+)>0.
	\]
	On the other hand,
	\[
	|g^{n_j}(W)|\to0
	\]
	forces the left-hand side to tend to zero.  This is a contradiction.
	Thus $g$ is not $C^0$-rigid, and restriction to the boundary shows that
	the disk map is not $C^0$-rigid.
\end{proof}

Figure~\ref{fig:interior-disk} summarizes interior stabilization and
the action-coordinate identification at the rigid seam.

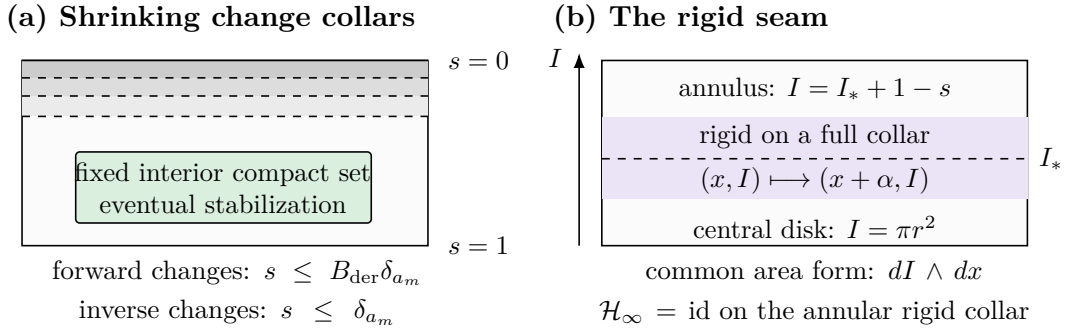
\begin{figure}[!htbp]
\centering
\begin{tikzpicture}[x=1cm,y=1cm,>=Latex,font=\small]
  \node[anchor=west,font=\bfseries] at (0,4.10)
    {(a) Shrinking change collars};
  \draw[line width=.6pt,fill=gray!3] (.35,1.10) rectangle (5.72,3.55);
  \fill[gray!16] (.35,2.81) rectangle (5.72,3.55);
  \fill[gray!27] (.35,3.08) rectangle (5.72,3.55);
  \fill[gray!40] (.35,3.32) rectangle (5.72,3.55);
  \draw[line width=.6pt] (.35,3.55) -- (5.72,3.55);
  \draw[line width=.6pt,dashed] (.35,2.81) -- (5.72,2.81);
  \draw[line width=.6pt,dashed] (.35,3.08) -- (5.72,3.08);
  \draw[line width=.6pt,dashed] (.35,3.32) -- (5.72,3.32);
  \node[anchor=west] at (5.85,3.55) {$s=0$};
  \node[anchor=west] at (5.85,1.10) {$s=1$};
  \draw[line width=.6pt,fill=coregreen,rounded corners=1.5pt]
    (1.06,1.40) rectangle (4.97,2.34);
  \node[align=center] at (3.01,1.87)
    {fixed interior compact set\\eventual stabilization};
  \node[align=center,text width=6.6cm] at (3.20,.48)
    {forward changes: $s\le B_{\mathrm{der}}\delta_{a_m}$\\[3pt]
     inverse changes: $s\le\delta_{a_m}$};

  \node[anchor=west,font=\bfseries] at (7.24,4.10)
    {(b) The rigid seam};
  \draw[line width=.6pt,fill=gray!3] (8.02,1.10) rectangle (13.65,3.55);
  \fill[collarviolet] (8.02,1.72) rectangle (13.65,2.80);
  \draw[line width=.6pt,dashed] (8.02,2.25) -- (13.65,2.25);
  \draw[-{Latex[length=2mm]},line width=.6pt] (7.72,1.10) -- (7.72,3.65);
  \node[anchor=east] at (7.63,3.58) {$I$};
  \node[anchor=west,inner sep=0pt] at (13.80,2.25) {$I_*$};
  \node at (10.84,3.17) {annulus: $I=I_*+1-s$};
  \node at (10.84,1.35) {central disk: $I=\pi r^2$};
  \node[align=center] at (10.84,2.55)
    {rigid on a full collar};
  \node at (10.84,1.97) {$(x,I)\longmapsto(x+\alpha,I)$};
  \node[align=center,text width=6.5cm] at (10.83,.48)
    {common area form: $dI\wedge dx$\\[3pt]
     $\mathcal H_\infty=\id$ on the annular rigid collar};
\end{tikzpicture}
\caption{Interior stabilization and disk completion. The change regions
of the conjugacies and their inverses shrink to the outer boundary, while
each interior compact set eventually stays unchanged. Their stabilized
limits give the smooth interior conjugacy. At the inner seam $s=1$,
the annular and central-disk action coordinates match at $I=I_*$;
the maps and area forms agree on a full collar. The right panel depicts
local action charts, not the physical invariant foliation.}
\label{fig:interior-disk}
\end{figure}

\Needspace{11\baselineskip}
\begin{remark}\label{rem:further-consequences}
The finite-stage annular conjugacies are the identity on a fixed inner
collar and therefore extend smoothly by the identity over the central
disk.  Under the fixed normalization used in the proof of
Proposition~\ref{prop:master-convergence}, the resulting smooth
area-preserving disk conjugates
$\widehat F_m=\widehat{\mathcal H}_m^{-1}\mathcal R_\alpha
\widehat{\mathcal H}_m$
converge in $C^1$ to the non-$C^0$-rigid disk pseudo-rotation $F$.
Thus the $C^1$ closure of the smooth area-preserving conjugacy class of
$\mathcal R_\alpha$, taken in
$\mathrm{Diff}^1_{\omegaZero}(\D)$, contains a non-$C^0$-rigid map.
Since every conjugate of the rotation is $C^0$-rigid, its $C^1$
area-preserving conjugacy class is also not closed.  This concerns closure
in the ambient $C^1$ space and does not answer the $C^\infty$-closedness
question for Diophantine $\alpha$ in
\cite[Question~7]{FayadKrikorian}.
\end{remark}

\section{Rotation-number schedules}\label{sec:arithmetic-schedules}

We use the representative $0<\alpha<1$ and number the convergents by
$p_0/q_0=0/1$.  For every $n\ge1$ we use
\begin{equation}\label{eq:conv-est}
 \frac1{q_n+q_{n+1}}<\normT{q_n\alpha}<\frac1{q_{n+1}}
\end{equation}
and the best-approximation property
\begin{equation}\label{eq:best-approx}
  0<k<q_{n+1},\quad k\in\N
  \quad\Longrightarrow\quad
 \normT{k\alpha}\ge\normT{q_n\alpha}.
\end{equation}

The next proposition supplies both the Brjuno $C^1$ schedule and the
weighted H\"older schedule.

\begin{proposition}\label{prop:brjuno-batched-schedule}
	Let $\alpha\in\Bclass$ and fix $0<\vartheta<1/40$.
	Choose $C_{\mathrm{chk}}>0$ sufficiently large, depending only on the
	fixed one-dimensional squeeze data chosen in
	Lemma~\ref{lem:protected-squeeze}, and set
	\begin{equation}\label{eq:checkpoints}
		\mathfrak k_j=
		\left\lceil
		\log_2q_j+2\log_2(\log_2q_j+2)+C_{\mathrm{chk}}
		\right\rceil.
	\end{equation}
	After discarding a finite initial segment, list in increasing order the
	indices $j$ for which $\mathfrak k_{j+1}>\mathfrak k_j$ as $j=j(m)$, and put
	\[
	a_m=\mathfrak k_j,\qquad b_m=\mathfrak k_{j+1},\qquad
	D_m=\lfloor\vartheta q_j\rfloor,\qquad
	T_m=T_{m-1}+D_m,
	\]
	where $T_{-1}=0$.  Consecutive retained indices satisfy
	$\mathfrak k_{j(m)+1}=\mathfrak k_{j(m+1)}$, because every intervening
	checkpoint belongs to a zero-width block; hence $b_m=a_{m+1}$.
	By discarding a sufficiently long finite initial segment, the first
	retained index $j(0)$, and hence the initial scale
	$a_0=\mathfrak k_{j(0)}$, can be chosen arbitrarily large.
	Then \eqref{eq:tower-separation} holds.
	
	If the allocation exponent is $\eta=0$, then
	\begin{equation}\label{eq:brjuno-block-cost}
		\Xi_m(0)\le
		C_{\mathrm{Br}}\frac{1+\log q_{j+1}}{q_j},
		\qquad \sum_m\Xi_m(0)<\infty,
	\end{equation}
	where $C_{\mathrm{Br}}>0$ depends only on $\vartheta$ and the fixed
	one-dimensional squeeze data of
	Lemma~\ref{lem:protected-squeeze}.
	
	If $0<\beta<1$, $\alpha\in\Bclassbeta{\beta}$, and $\eta=\beta$, then
	\begin{equation}\label{eq:weighted-block-cost}
		\Xi_m(\beta)\le C_{\mathrm{WB},\beta}
		\frac{q_{j+1}^{\beta}(1+\log q_{j+1})^{2\beta}}{q_j},
		\qquad \sum_m\Xi_m(\beta)<\infty,
	\end{equation}
	where $C_{\mathrm{WB},\beta}>0$ depends only on $\beta$, $\vartheta$,
	and the fixed one-dimensional squeeze data of
	Lemma~\ref{lem:protected-squeeze}.
\end{proposition}

\begin{proof}
	Let $C$ be the constant in Lemma~\ref{lem:protected-squeeze}.  Put
	\[
	C_{\log}:=C_{\mathrm{chk}}+4.
	\]
	Choose $C_{\mathrm{chk}}$ sufficiently large that
	\[
	2(C+3)C_{\log}^2\,2^{-C_{\mathrm{chk}}}<\frac12.
	\]
	Such a choice is possible since
	$(C_{\mathrm{chk}}+4)^2 2^{-C_{\mathrm{chk}}}\to0$ as
	$C_{\mathrm{chk}}\to\infty$.
	
	The definition \eqref{eq:checkpoints} gives
	\begin{equation}\label{eq:checkpoint-bounds}
		2^{-\mathfrak k_j}
		\le
		\frac{2^{-C_{\mathrm{chk}}}}
		{q_j(\log_2q_j+2)^2},
		\qquad
		\mathfrak k_j+1
		\le
		C_{\log}(\log_2q_j+2).
	\end{equation}
	After discarding a finite initial segment, the function
	$r\mapsto(r+1)^2 2^{-r}$ is decreasing throughout the relevant range.
	Since $A_r=Cr^2+3$ and $\delta_r=2^{-r}$, for every
	$r\ge\mathfrak k_j$ we therefore have
	\begin{align}
		2A_r\delta_r
		&\le
		2(C+3)(r+1)^2 2^{-r} \notag\\
		&\le
		2(C+3)(\mathfrak k_j+1)^2 2^{-\mathfrak k_j}
		<\frac1{2q_j},
		\label{eq:checkpoint-geometric-bound}
	\end{align}
	where the last inequality follows from
	\eqref{eq:checkpoint-bounds} and the choice of
	$C_{\mathrm{chk}}$.  Thus the checkpoint scale
	$\mathfrak k_j$ is chosen sufficiently fine that the corresponding
	local support size is strictly smaller than the continued-fraction
	separation scale $q_j^{-1}$.
	
	We next control the widths of the nonempty blocks.  The sequence
	$(\mathfrak k_j)$ is nondecreasing, and the retained-index rule gives
	$b_m=a_{m+1}$.  For a retained index $j=j(m)$,
	\[
	w_m=b_m-a_m
	\le b_m=\mathfrak k_{j+1}.
	\]
	Using \eqref{eq:checkpoint-bounds} with $j+1$, and changing from
	$\log_2$ to the natural logarithm, there is a constant
	$C_{\mathrm{width}}>0$, depending only on the fixed choice of
	$C_{\mathrm{chk}}$, such that
	\begin{equation}\label{eq:checkpoint-width}
		w_m
		\le
		C_{\mathrm{width}}(1+\log q_{j+1}).
	\end{equation}
	Since $\alpha\in\Bclass$ and the denominators $q_j$ have Fibonacci
	lower growth,
	\[
	\sum_j\frac{1+\log q_{j+1}}{q_j}<\infty.
	\]
	In particular,
	\[
	\frac{w_m}{q_j}
	\le
	C_{\mathrm{width}}
	\frac{1+\log q_{j+1}}{q_j}
	\longrightarrow0
	\]
	along the retained indices.
	
Since $(\mathfrak k_j)$ is nondecreasing and tends to infinity, there
are infinitely many retained indices.  As $b_m=a_{m+1}$, the
corresponding nonempty scale blocks are consecutive and cover all
integer scales $r\ge a_0$.  By
	discarding a sufficiently long finite initial segment, we may therefore
	assume simultaneously that
	\[
	a_0\ge K_{\min},
	\qquad
	j\ge2,
	\qquad
	D_m\ge2w_m,
	\]
	and, since $D_m=\lfloor\vartheta q_j\rfloor$,
	\begin{equation}\label{eq:fresh-height-bound}
		D_m\ge\frac{\vartheta q_j}{2}.
	\end{equation}
	Indeed, the last inequality holds for all sufficiently large $q_j$,
	while $D_m\ge2w_m$ follows from
	$w_m/q_j\to0$.
	
	It remains to verify the separation of the rotation tower.  Since
	the retained indices form a subsequence of the continued-fraction
	indices,
	\[
	T_m
	=
	\sum_{i=0}^{m}D_i
	\le
	\vartheta\sum_{n\le j}q_n.
	\]
	The recurrence $q_{n+2}\ge2q_n$ gives
	\[
	\sum_{n\le j}q_n\le4q_j,
	\]
	and hence
	\[
	T_m\le4\vartheta q_j.
	\]
	Because $\vartheta<1/40$, after one further finite truncation we have
	\[
	2T_m+4<q_j.
	\]
	Thus, if $1\le n\le2T_m+4$, then $n<q_j$.  Applying
	\eqref{eq:best-approx} with the continued-fraction index $j-1$ gives
	\[
	\normT{n\alpha}
	\ge
	\normT{q_{j-1}\alpha}.
	\]
	By \eqref{eq:conv-est},
	\[
	\normT{q_{j-1}\alpha}
	>
	\frac1{q_{j-1}+q_j}
	>
	\frac1{2q_j}.
	\]
	Since $a_m=\mathfrak k_j$, the estimate
	\eqref{eq:checkpoint-geometric-bound} gives
	\[
	\frac1{2q_j}
	>
	2A_{a_m}\delta_{a_m}.
	\]
	Consequently,
	\[
	\min_{1\le n\le2T_m+4}\normT{n\alpha}
	>
	2A_{a_m}\delta_{a_m}.
	\]
	Together with $D_m\ge2w_m$, this proves
	\eqref{eq:tower-separation}.
	
	We now estimate the block costs.  First take $\eta=0$.  Then
	\[
	U_m(0)=w_m.
	\]
	By \eqref{eq:Xi-definition} and
	\eqref{eq:allocation-bounds},
	\[
	\Xi_m(0)
	\le
	\frac{2U_m(0)}{D_m}
	=
	\frac{2w_m}{D_m}.
	\]
	Using \eqref{eq:fresh-height-bound} and
	\eqref{eq:checkpoint-width}, we obtain
	\[
	\Xi_m(0)
	\le
	\frac{4C_{\mathrm{width}}}{\vartheta}
	\frac{1+\log q_{j+1}}{q_j}.
	\]
	After absorbing the fixed change of logarithm into the constant, this
	is \eqref{eq:brjuno-block-cost}.  Since the retained indices form a
	subsequence of the continued-fraction indices, its upper-bound series
	is dominated by
	\[
	\sum_j\frac{1+\log q_{j+1}}{q_j},
	\]
	which converges by the Brjuno hypothesis together with the Fibonacci
	lower growth of $q_j$.  Hence
	\[
	\sum_m\Xi_m(0)<\infty.
	\]
	
	Finally, let $\eta=\beta$.  Since $\delta_r=2^{-r}$,
	\[
	U_m(\beta)
	=
	\sum_{r=a_m}^{b_m-1}2^{\beta r}
	\le
	C_{\mathrm{geom},\beta}\,2^{\beta b_m},
	\qquad
	C_{\mathrm{geom},\beta}
	:=
	\frac{2^\beta}{2^\beta-1}.
	\]
	Since $b_m=\mathfrak k_{j+1}$, the checkpoint definition
	\eqref{eq:checkpoints} gives
	\[
	2^{b_m}
	\le
	2^{C_{\mathrm{chk}}+1}
	q_{j+1}(\log_2q_{j+1}+2)^2.
	\]
	Therefore, using again \eqref{eq:allocation-bounds} and
	\eqref{eq:fresh-height-bound},
	\begin{align*}
		\Xi_m(\beta)
		&\le
		\frac{2U_m(\beta)}{D_m}\\
		&\le
		C_{\beta,\vartheta,C_{\mathrm{chk}}}
		\frac{
			q_{j+1}^{\beta}
			(\log_2q_{j+1}+2)^{2\beta}
		}{q_j}\\
		&\le
		C_{\mathrm{WB},\beta}
		\frac{
			q_{j+1}^{\beta}
			(1+\log q_{j+1})^{2\beta}
		}{q_j}.
	\end{align*}
	This is \eqref{eq:weighted-block-cost}.  The constant
	$C_{\mathrm{WB},\beta}$ depends only on $\beta$, $\vartheta$, and the
	fixed one-dimensional squeeze data through the fixed choice of
	$C_{\mathrm{chk}}$.  The upper-bound series converges by
	\eqref{eq:weighted-Brjuno-class}, and hence
	\[
	\sum_m\Xi_m(\beta)<\infty.
	\]
\end{proof}

\Needspace{7\baselineskip}
We record the relation with the more familiar denominator-growth exponent.

\begin{lemma}\label{lem:nu-growth}
Let $\alpha\in\T\setminus\mathbb Q$ and $0<\beta<1$.
\begin{enumerate}[label=\textnormal{(\roman*)}]
\item If $\nu(\alpha)<1/\beta$, then $\alpha\in\Bclassbeta{\beta}$.
\item One has $\nu(\alpha)=1$ for Lebesgue almost every irrational
$\alpha\in\T$.
\end{enumerate}
\end{lemma}

\begin{proof}
For part~\textnormal{(i)}, choose
$\nu(\alpha)<\tau<1/\beta$.  Eventually $q_{n+1}\le q_n^\tau$, and hence
\[
 \frac{q_{n+1}^{\beta}(1+\log q_{n+1})^{2\beta}}{q_n}
 \le C_{\tau,\beta}
 q_n^{\tau\beta-1}(1+\log q_n)^{2\beta}.
\]
Here $C_{\tau,\beta}>0$ is independent of $n$.  The right-hand side is
summable because $\tau\beta-1<0$ and $q_n$ grows
at least at the Fibonacci rate.

For part~\textnormal{(ii)}, fix $\gamma>0$ and let $k$ range over the
positive integers.  The set
\[
 \{\alpha\in\T:\normT{k\alpha}<k^{-1-\gamma}\}
\]
has measure $2k^{-1-\gamma}$ for all large $k$.  Borel--Cantelli and
$\normT{q_n\alpha}<q_{n+1}^{-1}$ imply
$q_{n+1}<q_n^{1+\gamma}$ eventually for almost every $\alpha$.  Intersecting
over $\gamma=1/s$, $s=1,2,\ldots$, gives $\nu(\alpha)\le1$ almost
everywhere; the reverse inequality follows from $q_{n+1}>q_n>1$ for all
sufficiently large $n$.
\end{proof}

\Needspace{18\baselineskip}
Figure~\ref{fig:arithmetic-schedule} displays the two arithmetic
allocations and their inputs to the master convergence criterion.

\begin{figure}[H]
\centering
\begin{tikzpicture}[x=1cm,y=1cm,>=Latex,font=\small]
  \node[align=center,text width=14.0cm] at (7.05,4.10)
    {$\displaystyle\mathfrak k_j=
      \left\lceil\log_2q_j+2\log_2(\log_2q_j+2)+C_{\mathrm{chk}}\right\rceil$};
  \node[draw=black!65,fill=gray!7,rounded corners=1.5pt,
    inner sep=4pt,align=center,text width=8.3cm] (blocks) at (7.05,3.29)
    {retain nonempty blocks: $a_m=\mathfrak k_j$, $b_m=\mathfrak k_{j+1}$\\
     $D_m=\lfloor\vartheta q_j\rfloor$; separated towers and $D_m\ge2w_m$};
  \node[draw=black!65,fill=softblue,rounded corners=1.5pt,
    inner sep=5pt,align=center,text width=6.0cm,minimum height=1.5cm]
    (c1) at (3.37,1.57)
    {$\eta=0$, $\alpha\in\Bclass$\\[3pt]
     $U_m(0)=w_m$\\[4pt]
     $\sum_m\Xi_m(0)<\infty\quad\Longrightarrow\quad C^1$};
  \node[draw=black!65,fill=coregreen,rounded corners=1.5pt,
    inner sep=5pt,align=center,text width=6.0cm,minimum height=1.5cm]
    (holder) at (10.74,1.57)
    {$\eta=\beta$, $\alpha\in\Bclassbeta{\beta}$, $0<\beta<1$\\[3pt]
     $U_m(\beta)\lesssim2^{\beta b_m}$\\[4pt]
     $\sum_m\Xi_m(\beta)<\infty
        \quad\Longrightarrow\quad C^{1,\beta}$};
  \draw[-{Latex[length=2mm]},line width=.6pt]
    ([xshift=-2.25cm]blocks.south) -- (c1.north);
  \draw[-{Latex[length=2mm]},line width=.6pt]
    ([xshift=2.25cm]blocks.south) -- (holder.north);
  \node[align=center,text width=14.0cm] at (7.05,.15)
    {$\nu(\alpha)<1/\beta\ \Longrightarrow\ \alpha\in\Bclassbeta{\beta}$;
     \qquad $\nu(\alpha)=1/\beta$ alone is inconclusive.};
\end{tikzpicture}
\caption{Two separate arithmetic allocations. The checkpoint correction
separates towers; retaining nonempty blocks preserves consecutive scales.
Each branch fixes its allocation and tests $\Xi_m(t)$ at $t=\eta$, with
$\Xi_m(\eta)\le2U_m(\eta)/D_m$. The respective Brjuno sums yield the
convergence criterion of Proposition~\ref{prop:master-convergence}.
The strict growth bound is sufficient; critical equality permits both
weighted membership and nonmembership.}
\label{fig:arithmetic-schedule}
\end{figure}
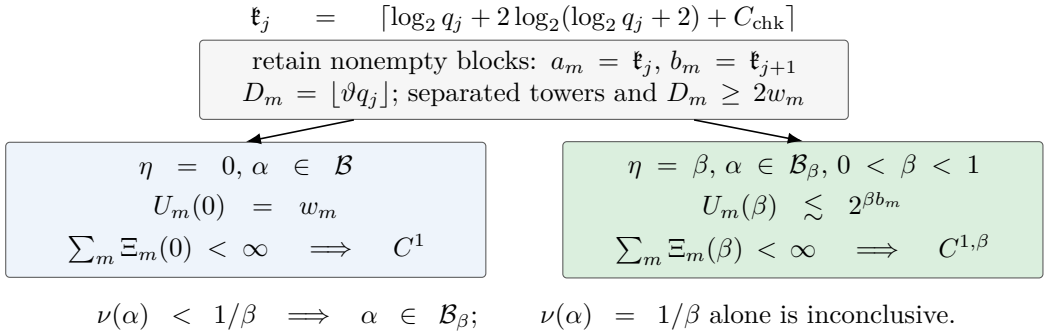

\begin{remark}\label{rem:brjuno-liouville}
The Brjuno $C^1$ theorem includes genuinely Liouville rotation numbers.
For example, after fixing initial continued-fraction data, choose the next
partial quotient recursively by
\[
 \mathfrak a_{n+1}=\left\lceil e^{\sqrt{q_n}}\right\rceil.
\]
Then
\[
 \log q_{n+1}=\sqrt{q_n}+\log q_n+O(1),
\]
so
\[
 \sum_n\frac{\log q_{n+1}}{q_n}<\infty
\]
by the Fibonacci lower growth of $q_n$.  On the other hand,
$q_{n+1}/q_n^N\to\infty$ for every fixed $N$, so the resulting irrational is
Liouville.  Thus the Brjuno range in Theorem~\ref{thm:main}\textnormal{(i)}
is not confined to finite-type arithmetic.
\end{remark}

For completeness, both comparisons in Remark~\ref{rem:rotation-range}
can be realized by actual continued fractions.  For fixed $0<\beta<1$,
choose $\mathfrak a_{n+1}$, after a finite initial segment, to be a nearest
positive integer to $q_n^{1/\beta-1}$.  Then
$q_{n+1}=q_n^{1/\beta}(1+o(1))$.
The Brjuno terms are bounded by a constant times $\log q_n/q_n$, whereas
the weighted terms are comparable to $(\log q_n)^{2\beta}$ and do not
even tend to zero.  This proves $\Bclassbeta{\beta}\subsetneq\Bclass$.
The critical example in that remark can be made explicit as follows.  Write $\mathfrak a_{n+1}$ for the next partial
quotient.  Once $q_n$ is large, choose $\mathfrak a_{n+1}$
to be the nearest positive integer to
\[
 q_n^{1/\beta-1}\exp\!\bigl(- (\log q_n)^{3/4}\bigr).
\]
Put $Q_n=q_n^{1/\beta}\exp\!\bigl(- (\log q_n)^{3/4}\bigr)$ and
$E_n=Q_n/q_n$.  Since $0<\beta<1$, we have $E_n\to\infty$.
Rounding changes $\mathfrak a_{n+1}$ by at most $1/2$, and
$0\le q_{n-1}\le q_n$; hence
\[
 \left|\frac{q_{n+1}}{Q_n}-1\right|
 \le \frac{3}{2E_n}\longrightarrow0.
\]
In particular,
\[
 q_{n+1}\asymp q_n^{1/\beta}
 \exp\!\bigl(- (\log q_n)^{3/4}\bigr),
\]
with constants independent of $n$.  Taking logarithms gives
$\log q_{n+1}=\beta^{-1}\log q_n-(\log q_n)^{3/4}+o(1)$, so the
logarithmic quotient tends to $1/\beta$.  The $n$-th term of
$B_\beta(\alpha)$ is bounded by
\[
 C_{\mathrm{crit},\beta}(1+\log q_n)^{2\beta}
 \exp\!\bigl(-\beta(\log q_n)^{3/4}\bigr),
\]
where $C_{\mathrm{crit},\beta}>0$ is independent of $n$.  This is summable
by Fibonacci growth.

\begin{proof}[Proof of Theorem~\ref{thm:local-nonrigidity}]
For part~\textnormal{(i)}, start the construction at an arbitrarily late
checkpoint $\mathfrak k_J$ and choose the near-equal allocation exponent
$\eta=0$.  Since the initial disk map is the rigid rotation,
Proposition~\ref{prop:brjuno-batched-schedule},
Corollary~\ref{cor:holder-stage}, and the fixed disk norm comparison give a
constant $C_{\mathrm{loc},0}>0$, independent of $J$, such that
\[
 \|F-\mathcal R_\alpha\|_{C^1(\D)}
 \le C_{\mathrm{loc},0}
 \sum_{j\ge J}\frac{1+\log q_{j+1}}{q_j}.
\]
The right-hand side tends to zero because $\alpha$ is Brjuno.

For part~\textnormal{(ii)}, start at the same kind of late checkpoint and
choose $\eta=\beta$.  There is a constant
$C_{\mathrm{loc},\beta}>0$, independent of $J$, such that
\[
 \|F-\mathcal R_\alpha\|_{C^{1,\beta}(\D)}
 \le C_{\mathrm{loc},\beta}
 \sum_{j\ge J}
 \frac{q_{j+1}^{\beta}(1+\log q_{j+1})^{2\beta}}{q_j}.
\]
This tail tends to zero because $\alpha\in\Bclassbeta{\beta}$.
In either case choose $J$ so that the corresponding bound is less than
$\varepsilon$.  Proposition~\ref{prop:master-convergence} supplies the
pseudo-rotation, Denjoy boundary dynamics, unique periodic center, smooth
interior foliation, and failure of $C^0$-rigidity.
\end{proof}

\begin{proof}[Proof of Corollary~\ref{cor:local-nonrigidity-nu}]
Lemma~\ref{lem:nu-growth}\textnormal{(i)} gives
$\alpha\in\Bclassbeta{\beta}\subset\Bclass$, so the first assertion follows from
Theorem~\ref{thm:local-nonrigidity}\textnormal{(ii)}.
Part~\textnormal{(ii)} gives the almost-everywhere statement.  
\end{proof}

\begin{proof}[Proof of Theorem~\ref{thm:main}]
For part~\textnormal{(i)}, choose $\eta=0$.  Equations
\eqref{eq:brjuno-block-cost} and \eqref{eq:C1-master-condition} give the
$C^1$ example.  For part~\textnormal{(ii)}, fix $0<\beta<1$ with
$\alpha\in\Bclassbeta{\beta}$ and choose $\eta=\beta$.  Then
\eqref{eq:weighted-block-cost} and
\eqref{eq:Holder-master-condition} give the $C^{1,\beta}$ example.  The
remaining conclusions are those of Proposition~\ref{prop:master-convergence}.
\end{proof}

\begin{proof}[Proof of Corollary~\ref{cor:superpolynomial-growth}]
Choose the Brjuno--Liouville number from
Remark~\ref{rem:brjuno-liouville}, so that
\[
 \log q_{n+1}=\sqrt{q_n}+\log q_n+O(1).
\]
In particular, $q_{n+1}/q_n^N\to\infty$ for every fixed $N>0$.
Let $F$ be the non-$C^0$-rigid $C^1$ pseudo-rotation supplied by
Theorem~\ref{thm:main}\textnormal{(i)} for this one fixed $\alpha$.

Fix $A>0$.  If the asserted limsup were finite, its eventual bound,
together with the finitely many initial iterates, would give a constant
$C_A>0$ such that
\[
 \|DF^m\|_{C^0(\D)}\le C_A m^A
 \qquad(m\ge1).
\]
Pass to the infinite parity subsequence of continued-fraction
convergents for which
\[
 0<\varepsilon_n:=q_n\alpha-p_n<q_{n+1}^{-1}.
\]
\Needspace{11\baselineskip}
The map $F^{q_n}$ is again a pseudo-rotation: a periodic point of an
iterate would be a periodic point of $F$.  Its rotation number is
$q_n\alpha+\Z=\varepsilon_n+\Z$.  The differentiable form of
\cite[Corollary~B]{AFLXZ}, with any fixed geometric norm-normalization
factor absorbed into $C_{\mathrm{disp}}>0$, gives
\begin{align*}
 \|F^{q_n}-\id\|_{C^0(\D)}
 &\le C_{\mathrm{disp}}
       \bigl(1+\|DF^{q_n}\|_{C^0(\D)}\bigr)\sqrt{\varepsilon_n}\\
 &\le C_{\mathrm{disp}}
       \frac{1+C_Aq_n^A}{\sqrt{q_{n+1}}}
 \longrightarrow0.
\end{align*}
The constant $C_{\mathrm{disp}}$ is independent of $n$; the last limit
uses $q_{n+1}/q_n^{2A}\to\infty$.
Thus $F^{q_n}$ converges uniformly to the identity along this subsequence,
contradicting the choice of $F$.  Since the same map $F$ was fixed before
$A$, the conclusion holds for every $A>0$.
\end{proof}

\section*{Acknowledgements}
The authors thank Prof. Zhihong Xia for helpful questions and comments. We acknowledge the use of ChatGPT for computations of certain details and for language proofreading. The central idea, the overall structure of the proof, and the final formulation of the results were developed entirely by the authors. The correctness of the computational details performed by the AI has been manually verified by the authors, who take full responsibility for this paper.


\begin{thebibliography}{99}

\bibitem{AFLXZ}
A.~Avila, B.~Fayad, P.~Le Calvez, D.~Xu, and Z.~Zhang,
\emph{On mixing diffeomorphisms of the disc},
Invent. Math. \textbf{220} (2020), no.~3, 673--714.
\href{https://doi.org/10.1007/s00222-019-00937-7}{doi:10.1007/s00222-019-00937-7}.


\bibitem{BramhamRigid}
B.~Bramham,
\emph{Pseudo-rotations with sufficiently Liouvillean rotation number are $C^0$-rigid},
Invent. Math. \textbf{199} (2015), 561--580.
\href{https://doi.org/10.1007/s00222-014-0525-0}{doi:10.1007/s00222-014-0525-0}.

\bibitem{BramhamZhangSurvey}
B.~Bramham and Z.~Zhang,
\emph{Recent results and open questions on pseudo-rotations},
in \emph{A Vision for Dynamics in the 21st Century: The Legacy of Anatole Katok},
Cambridge Univ. Press, 2024, 67--93.
\href{https://doi.org/10.1017/9781009278898.004}{doi:10.1017/9781009278898.004}.


\bibitem{FayadKrikorian}
B.~Fayad and R.~Krikorian,
\emph{Some questions around quasi-periodic dynamics},
in \emph{Proceedings of the International Congress of Mathematicians--Rio de Janeiro 2018, Vol.~III},
World Scientific, 2018, 1909--1932.
\href{https://doi.org/10.1142/9789813272880_0123}{doi:10.1142/9789813272880\_0123}.

\bibitem{FayadKrikorianHerman}
B.~Fayad and R.~Krikorian,
\emph{Herman's last geometric theorem},
Ann. Sci. \'Ec. Norm. Sup\'er. (4) \textbf{42} (2009), no.~2, 193--219.
\href{https://doi.org/10.24033/asens.2093}{doi:10.24033/asens.2093}.

\bibitem{Herman}
M.~R.~Herman,
\emph{Sur la conjugaison diff\'erentiable des diff\'eomorphismes du cercle \`a des rotations},
Publ. Math. Inst. Hautes \'Etudes Sci. \textbf{49} (1979), 5--233.
\href{https://doi.org/10.1007/BF02684798}{doi:10.1007/BF02684798}.

\bibitem{Russmann}
H.~R\"ussmann,
\emph{Stability of elliptic fixed points of analytic area-preserving mappings under the Bruno condition},
Ergodic Theory Dynam. Systems \textbf{22} (2002), no.~5, 1551--1573.
\href{https://doi.org/10.1017/S0143385702000974}{doi:10.1017/S0143385702000974}.


\end{thebibliography}
\end{document}